\documentclass[11pt,a4paper]{article}
\usepackage{amsmath,amsfonts,amssymb,mathrsfs}
\usepackage{graphicx} 
\usepackage{subfigure}
\usepackage{pstricks}
\usepackage{pst-node}
\usepackage{xcolor}
\usepackage{cite}
\usepackage[T1]{fontenc}
\usepackage{mathscinet}
\usepackage[english]{babel}
\usepackage{soul}
\usepackage{hyperref}
\usepackage{mathtools}
\usepackage{wasysym}
\usepackage{caption}

\mathtoolsset{showonlyrefs=true}

\newrgbcolor{blu}{.1 0 .7}
\newrgbcolor{lgrey}{.9 .9 .9}
\newrgbcolor{grey}{.7 .7 .7}
\newrgbcolor{dgrey}{.3 .3 .3}
\newrgbcolor{lred}{1 .8 .8}
\newrgbcolor{mlred}{1 .7 .7}
\newrgbcolor{bblue}{.6 .6 1}
\newrgbcolor{lblue}{.8 .8 1}
\newrgbcolor{mblue}{.2 .25 1}
\newrgbcolor{dblue}{.3 .32 1}
\newrgbcolor{ddblue}{.4 .4 1}
\newrgbcolor{lgreen}{.8 1 .8}
\newrgbcolor{mgreen}{ .7 1 .7}
\newrgbcolor{ggreen}{ .1 .75 .35}
\newtheorem{theorem}{\bf Theorem}[section]
\newtheorem{lemma}[theorem]{\bf Lemma}

\newtheorem{definition}[theorem]{\bf Definition}

\newtheorem{proposition}[theorem]{\bf Proposition}

\numberwithin{equation}{section}

\newcommand{\R}{\mathbb{R}}

\newcommand{\N}{\mathbb{N}}

\def \L {\mathscr{L}}

\def \kk {{\widehat k}}

\def \loc {{\text{\rm loc}}}
\def \p {\partial}
\def \tilde {\widetilde}

\def \rho {\varrho}
\def \theta {\vartheta}

\newenvironment{sergiorev}{\color{blue}}{\color{black}}
\newenvironment{giuliorev}{\color{red}}{\color{black}}
\newcommand{\bsr}{\begin{sergiorev}}
\newcommand{\esr}{\end{sergiorev}}
\newcommand{\bgr}{\begin{giuliorev}}
\newcommand{\egr}{\end{giuliorev}}
\newenvironment{cancelrev}{\color{green}}{\color{black}}
\newcommand{\bcanr}{\begin{cancelrev}}
\newcommand{\ecanr}{\end{cancelrev}}
\newcommand{\denoterow}[1]{\rlap{\hspace{1em}$\leftarrow{}$ $j^{th}$}}

\def \D {{\Delta}}

\newenvironment{proof}{\noindent {\sc Proof.}}{\hfill $\square$}

\def \o {{\omega}}
\def \a {{\alpha}}
\def \b {{\beta}}

\def \d {{\delta}}
\def \e {{\varepsilon}}

\def \epsilon {{\varepsilon}}

\def \r {{\rho}}

\def \t {{\tau}}

\def \x {{\xi}}
\def \y {{\eta}}

\def \w {{\omega}}
\def \phi {{\varphi}}
\def \G {{\Gamma}}
\def \O {{\Omega}}

\definecolor{shadecolor}{rgb}{.9, .9, 1}
\newrgbcolor{blu}{.1 0 .7}
\newrgbcolor{lgrey}{.88 .88 .88}
\newrgbcolor{llgrey}{.94 .94 .94}
\newrgbcolor{grey}{.75 .75 .75}
\newrgbcolor{dgrey}{.3 .3 .3}
\newrgbcolor{lred}{1 .8 .8}
\newrgbcolor{mlred}{1 .7 .7}
\newrgbcolor{bblue}{.6 .6 1}
\newrgbcolor{lblue}{.8 .8 1}
\newrgbcolor{mblue}{.2 .25 1}
\newrgbcolor{dblue}{.3 .32 1}
\newrgbcolor{ddblue}{.4 .4 1}
\newrgbcolor{lgreen}{.8 1 .8}
\newrgbcolor{mgreen}{ .7 1 .7}
\newrgbcolor{ggreen}{ .1 .75 .35}

\usepackage{mathtools}

\title{A study of the Kuramoto model for synchronization phenomena based on degenerate Kolmogorov-Fokker-Planck equations}
\author{{\sc{G. Pecorella}
\thanks{Dipartimento di Scienze Fisiche, Informatiche e Matematiche, Universit\`{a} di Modena e Reggio Emilia, Via
Campi 213/b, 41125 Modena (Italy). E-mail: giulio.pecorella@unimore.it} \qquad 
\sc{S. Polidoro}
\thanks{Dipartimento di Scienze Fisiche, Informatiche e Matematiche, Universit\`{a} di Modena e Reggio Emilia, Via
Campi 213/b, 41125 Modena (Italy). E-mail: sergio.polidoro@unimore.it} \qquad 
\sc{C. Vernia}
\thanks{Dipartimento di Scienze Fisiche, Informatiche e Matematiche, Universit\`{a} di Modena e Reggio Emilia, Via
Campi 213/b, 41125 Modena (Italy). E-mail: cecilia.vernia@unimore.it}
}}
\date{}

\begin{document}

\maketitle

\begin{abstract}
We study a nonlinear partial differential equation that arises when introducing inertial effects in the Kuramoto model. Based on the known theory of degenerate Kolmogorov operators, we prove existence, uniqueness and a priori estimates of the solution to the relevant Cauchy problem. Moreover, a stable numerical operator, which is consistent with the degenerate Kolmogorov operator, is introduced in order to produce numerical solutions. Finally, numerical experiments show how the synchronization phenomena depend on the parameters of the Kuramoto model with inertia.
\end{abstract}

\bigskip
\normalsize
\tableofcontents

\section{Introduction}
One of the most fascinating phenomena in nature is the tendency to synchronization. It pervades nature at every scale from the nucleus to the cosmos, and even our heart exhibits this phenomenon. A cluster of about $10,000$ cells, called the sinoatrial node, generates the electrical rhythm that commands the rest of the heart to beat, and it must do so reliably, minute after minute, for about three billion beats in a lifetime. These cells are a collection of oscillators, i.e. entities that cycle automatically in more or less regular time intervals, and need to coordinate their rhythm in order to achieve some kind of synchronization.

Huygens was one of the pioneers in the study of synchronization. In February of 1665 the Dutch physicist was confined to his bedroom for several days, and observed a curious phenomenon: the two pendulum clocks in the room, separated by two feet, kept oscillating together without any variation; moreover, mixing up the swings of the pendulums the synchronization returned after some time. Intrigued by this phenomenon, he carried out several experiments, finding out that the synchronization could take place only if they could communicate in some way.

In the following centuries many mathematicians studied the problem, coming up with different models: some of these describe specific problems (see \cite{Pe} for the description of sinoatrial node synchronization), while others have been created with the aim of being general enough to  describe the dynamics of different systems, even relevant to very different scientific fields such as biology, physics, and social sciences as well. One of the most famous model that belongs to the latter group is the one proposed by Kuramoto, from which many variations have been devised. We refer to the article \cite{Sp} by Spigler for an overview of this model and some of its derivatives. We next describe the Kuramoto model and some of its modifications we are interested in, then we discuss our contributions on this subject.

\medskip

The Kuramoto model is a system of ordinary differential equations introduced in the late decades of the last century in \cite{Ku}, that describes the collective behaviour of a population that can be seen as a group of oscillators. Let's consider a population of $N$ oscillators, we denote as $\theta_k(t)$ the phase related to the $k$-th oscillator, and as $\Omega_k$ its natural frequency, which is an intrinsic parameter of the $k$-th oscillator. We suppose that these natural frequencies are distributed according to an unimodal distribution $g=g(\Omega)$ that is normalized and symmetric with respect to its mean frequency $\overline{\Omega}$. The governing equation is given by the following system of ordinary differential equations
\begin{equation}\label{kuramoto}
  \dot{\theta}_k = \Omega_k + \frac{K}{N} \sum_{j=1}^N \sin(\theta_j-\theta_k),\qquad\;k=1,...,N,
\end{equation} 
where $K$ is a real positive constant sizing the coupling strength between the oscillators. Note that the coupling in the above system is nonlinear, thus the ensuing phenomena is expected to be rather complicated. On the other hand, a mean-field coupling, which is present in \eqref{kuramoto}, significantly simplify the evolution of the system. In order to point out this coupling effect, we introduce the following complex number, called \textit{order parameter}

\begin{equation}\label{parameterorder}
  r_N(t)\mathrm{e}^{i\psi_N(t)} = \frac{1}{N} \sum_{j=1}^N \mathrm{e}^{i\theta_j(t)}.
\end{equation}
The phase $\psi_N(t)$ is the mean phase of the system. The modulus $|r_N(t)|$ provides us with a measure of the synchronization of the system. When $|r_N(t)|=0$ the system is fully incoherent, when $0 < |r_N(t)| < 1$ the system is partly synchronized, and $|r_N(t)|=1$ matches the full coherence state. 
Equation \eqref{kuramoto} can be rewritten using the order parameter
\begin{equation}\label{kuramean}
  \dot{\theta}_k = \Omega_k + K r_N \sin(\psi_N-\theta_k),\qquad\;k=1,...,N,
\end{equation}
which is the equation of an overdamped pendulum with torque $\Omega_k$ and restoring force proportional to $Kr_N$. This formulation simplifies the analytical treatment.

Despite the simplicity of this model, Wiesenfeld, Colet and Strogatz show in \cite{Jos1, Jos2} that \eqref{kuramean} describes some important physical phenomena such as the interaction of quasi-optical oscillators with a cavity and some Josephson arrays, and it gives us a starting point to study synchronization in generic oscillators.
Moreover, Kuramoto model shows how synchronization in a population of many coupled oscillators may occur, and the type of synchronization: indeed it may occur frequency synchronization, where each oscillator completes its cycle at the same time, phase synchronization, where each oscillator is at the same point in the cycle, or even a more complex scenario.

The Kuramoto model can be extended to a population of infinitely many oscillators, as it was described by Strogatz in \cite{S}. Let's consider a continuum of oscillators, and let $\rho=\rho(\theta,\O,t)$ be a density function describing the fraction of oscillators at phase $\theta$ with natural frequency $\O$ at time $t$. Then $\rho$ is a nonnegative function $2\pi$ periodic in $\theta$ and satisfies 
\begin{equation}
\int_{0}^{2\pi} \rho(\theta,\O,t)d\theta=1,\qquad\forall\; (\O,t) \in \R\times [0,+\infty[.
\end{equation}
The evolution of $\rho$ is described by the usual continuity equation
\begin{equation}\label{eq-cont}
  \frac{\p \rho}{\p t}= -\frac{\p}{\p \theta}(\rho v),
\end{equation}
which expresses conservation of oscillators of frequency $\O$. Here $v = v(\theta,\O, t)$ is the instantaneous velocity of an oscillator at position $\theta$, given that it has natural frequency $\O$, and is interpreted in the Eulerian sense. From \eqref{kuramean} we see that that velocity is 
\begin{equation}
  v(\theta,\Omega,t)= \Omega + K r(t) \sin(\psi(t)-\theta),
\end{equation}
where $r(t)$ and $\psi(t)$ follow from \eqref{parameterorder} by letting $N \to + \infty$ 
\begin{equation}\label{eq-phase}
  r(t) \mathrm{e}^{\mathrm{i}\psi(t)}= \int_0^{2 \pi}\int_\R \mathrm{e}^{\mathrm{i}\theta'}\rho(\theta',\Omega',t)g(\Omega')d\Omega' d\theta'.
\end{equation}
Combining \eqref{eq-cont} and\eqref{eq-phase} we obtain the following nonlinear partial integro-differential equation for the density $\rho$
\begin{equation}\label{continuum}
  \frac{\p \rho}{\p t}(\theta,\Omega,t) = -\frac{\p}{\p \theta}\left[\rho(\theta,\Omega,t) \left(\Omega + K \int_0^{2 \pi}\int_\R \sin(\theta'-\theta)\rho(\theta',\omega',t)g(\Omega')d\Omega' d\theta'\right)\right].
\end{equation}
Sakaguchi in \cite{Sa} modified the deterministic equations \eqref{kuramean} and \eqref{continuum} by adding noise to describe stochastic phenomena, in order to consider rapid stochastic fluctuations in the natural frequencies.
The governing equations \eqref{kuramean} for $N$ oscillators take now the form of a system of Langevin equations
\begin{equation}\label{saga}
   \dot{\theta}_k = \Omega_k +\xi_k+ \frac{K}{N} \sum_{j=1}^N \sin(\theta_j-\theta_k),\qquad\;k=1,...,N,
\end{equation} 
where $\xi (t) = D \, W_t$ being $(W_t)_{t \ge0}$ a $N$-dimensional Wiener process, with $D > 0$.

Then, Sakaguchi argued intuitively that since \eqref{saga} is a system of Langevin equations with mean-field coupling, as $N \rightarrow \infty$ the density $\rho(\theta, \omega, t)$ should satisfy the following Fokker–Planck equation
\begin{equation}\label{parabolica}
  \frac{\p \rho}{\p t}= D \frac{\p^2 \rho}{\p\theta^2}-\frac{\p}{\p \theta}\left[\rho \left(\Omega + K \int_0^{2 \pi}\int_\R \sin(\theta'-\theta)\rho(\theta',\Omega',t)g(\Omega')d\Omega' d\theta'\right)\right].
\end{equation}
Note that \eqref{parabolica} is a nonlinear parabolic partial integro-differential equation, that reduces to \eqref{continuum}  when $D=0$. We finally recall that linear stability of the incoherent state, concerning both the deterministic and the stochastic equations \eqref{continuum} and \eqref{parabolica}, respectively, have been studied by Strogatz and Mirollo in \cite{SM}. The Cauchy problem relevant to \eqref{parabolica} has been studied in recent years by Lavrentiev, Spigler and Tani in \cite{Spp1}, \cite{Spp2} and \cite{Spp3}.

Based on some ideas developed by Ermentrout in his studies about fireflies synchronization \cite{E}, Tanaka, Lichtenberg and Oishi in \cite{TLO} elaborated the following second order variation of the Kuramoto model
\begin{equation}\label{tanaka}
  m \ddot{\theta}_k + \dot{\theta}_k = \Omega_k + \frac{K}{N}\sum_{j=1}^N \sin(\theta_j-\theta_k) ,\qquad k=1,...,N,
\end{equation}
with $m>0$. We point out that the term $m \ddot{\theta}_k$ introduces an inertial effect in the model, in order to better explain some synchronization phenomena. Unlike the original Kuramoto model, the Ermentrout's second order model has the property to allow near-perfect phase synchronization, where the phase shift between synchronized oscillators is inversely proportional to the mass. 
The equation \eqref{tanaka} can be rewritten in terms of the order parameter \eqref{parameterorder} as follows
\begin{equation}
    m \ddot{\theta}_k + \dot{\theta}_k = \Omega_k +K r_N \sin(\psi_N-\theta_k),\qquad i=1,...,N,
\end{equation}
which is the governing equation for a single damped driven pendulum with torque $\Omega_k$.

Starting from the contributions due to Tanaka, Lichtenberg and Oishi \cite{TLO}, Acebroan and Spigler introduced a noise term $\xi_k=\xi_k(t)$ in \eqref{tanaka}, so that  they considered in \cite{AS} a system of second order Langevin equations that consider an inertial term
\begin{equation}\label{spig}
  m \ddot{\theta}_k + \dot{\theta}_k=\Omega_k + K r_N \sin(\psi_N-\theta_k) + \xi_k,\qquad k=1,...,N.
\end{equation}
If we finally set $\dot{\theta}_k=\omega_k$, and we let $N \to \infty$ as we did in \eqref{parabolica}, we find the following nonlinear degenerate parabolic integro-differential equation
\begin{equation}\label{ultraparabolica} 
 \frac{D}{m^2} \frac{\p^2\rho}{\p\omega^2} + \frac{1}{m}\frac{\p}{\p\omega} [(\omega - \Omega - K_\rho(\theta,t))\rho] - \omega \frac{\p \rho}{\p\theta} -\frac{\p\rho}{\p t} =0,
\end{equation} 
where
\begin{equation} \label{kurainte}
K_\rho(\theta,t) = K \int_\R \int_\R \int_0^{2\pi} g(\Omega')\sin(\theta'-\theta)\rho(\omega',\theta',\Omega',t)d\theta'd\omega'd\Omega'.
\end{equation}

In order to simplify the notation, in the sequel we set $D=1$ and $m=1$. For a given $\rho_0 = \rho_0(\omega,\theta,\Omega)$ we consider the following Cauchy problem
\begin{equation} \label{eq-Cauchy-pbm}
  \begin{cases} 
 \frac{\p^2\rho}{\p\omega^2} + \frac{\p}{\p\omega} [(\omega - \Omega - K_\rho(\theta,t))\rho] - \omega \frac{\p \rho}{\p\theta} -\frac{\p\rho}{\p t} =0,\qquad &
	 (\o,\theta,\O,t) \in \R^3 \times [0,T], \\
 \rho (\omega,\theta,\O,0)=\rho_0(\omega,\theta,\Omega), & (\omega,\theta,\Omega) \in \R^3.
  \end{cases}
\end{equation} 
The Cauchy problem \eqref{eq-Cauchy-pbm} has been addressed by Akhmetov, Lavrentiev and Spigler in \cite{Sp1}, \cite{Sp2}, \cite{Sp3} and \cite{Sp4} by using a parabolic regularization of the governing equation.

The main result proved in \cite{Sp3} is the existence and the uniqueness of a classical solution $\rho$ to \eqref{eq-Cauchy-pbm}. Since $(\omega,\theta) \mapsto \rho(\omega,\theta,\Omega, t)$ denotes the density of a probability measure, the authors of the afore-mentioned references focus on positive solutions with the property 
\begin{equation*}
 \int_{]0, 2\pi[ \times \R} \rho(\omega,\theta,\Omega, t) d\theta \, d\omega= 1,
 \qquad t \ge 0, \ \Omega \in \R.
\end{equation*}
Moreover, the solutions considered there have the following \emph{exponential decay property}

\begin{definition}\label{def-expdecay}
We say that a function $f:\R^4 \rightarrow \R$ has the \emph{exponential decay property} if there exist two positive constants $M,C$ such that:
\begin{equation}\label{eq-decay}
(\text {\bf E}) \hspace{3.7cm}    |f(\o,\theta,\Omega,t)|\leq  C\mathrm{e}^{-M\omega^2}, 
\qquad \forall (\o,\theta,\Omega,t) \in \R^4. \hspace{3.7cm}
\end{equation}
\end{definition}

The main result proved in \cite{Sp3} is the existence of a unique positive classical solution $\rho$ to the Cauchy problem \eqref{eq-Cauchy-pbm}, with the properties listed above, provided that $\rho_0$ and some of its derivatives have the exponential decay property.

In this work we address the Cauchy problem \eqref{eq-Cauchy-pbm} by using the theory of subelliptic equations in Lie groups as described in the survey article \cite{AP}. In particular, the first order differential operator  $Y=\o\p_\o - \o\p_\theta - \p_t$ is considered as a \emph{Lie derivative}, that is

\begin{definition} \label{Def-Lie}
Consider the integral curve $\exp(sY)(z)$ of $Y$ defined for $z=(\o,\theta,\O,t) \in \R^4$ as the unique solution to
\begin{equation} \label{eq-expY}
\begin{cases}
 \frac{d}{d s} \exp(sY)(z) = Y \exp(sY)(z), \\
 \exp(sY)(z)\vert_{s=0}= (z).
\end{cases}
\end{equation}
A function $\rho$ is Lie differentiable with respect to $Y$ at the point $z=(\o,\theta,\O,t)$ if the following limit
\begin{equation}
  Y\rho(z)=\frac{d}{d s} \rho \left(\exp(sY)(z)\right)\big|_{s=0}
\end{equation}
exists and is finite.
\end{definition}
This definition clarifies the meaning of \emph{classical solution} to \eqref{ultraparabolica}. A function $\rho$ is a classical solution if it is continuous, its derivatives $\p_{\o}\rho$, and $\p^2_{\o\o}\rho$ are defined as continuous functions, the Lie derivative $Y \rho$ is continuous and the equation \eqref{ultraparabolica} is satisfied at every point. 

\medskip

The main advantage in our approach, with respect to \cite{Sp3}, is that we don't require unnecessary regularity on the data $\rho_0$. Moreover, our hypothesis \eqref{hp-g} below does not require the compactness of the support of $g$ assumed in\cite{Sp3}. Finally, a numerical approximation of the Lie derivative $Y\rho$ allows us to define a \emph{stable} numerical method for the Cauchy problem \eqref{eq-Cauchy-pbm}. Our first result is the following

\begin{theorem} \label{Th-main}
Let $g$ be a non-negative function such that
\begin{equation}
	\int_{\R}g(\O)d\O=1,
\end{equation}
and assume that exists $\b>2$ such that
\begin{equation}\label{hp-g}
	\int_{\R}g(\O)\mathrm{e}^{ |\O|^\b}	d\O <+\infty.
\end{equation}
Let $\rho_0 \in C(\R^3)$ be a strictly positive function, $2\pi$ periodic in $\theta$ such that
\begin{enumerate}
  \item $\rho_0$ verifies the exponential decay property {\rm ({\bf E})} as stated in Definition \ref{def-expdecay};
  \item for every $\Omega \in \R$ we have
  \begin{equation}
    \int_{]0, 2\pi[ \times \R} \rho_0(\o,\theta,\Omega) d\theta d\omega= 1.
  \end{equation}
\end{enumerate}
Then there exists a strictly positive classical solution $\rho$ to the Cauchy problem \eqref{eq-Cauchy-pbm}, defined for $(\o,\theta,\Omega,t) \in \R^3 \times [0, + \infty[$ such that
\begin{equation}
  \int_{]0, 2\pi[ \times \R} \rho(\o,\theta,\Omega,t) d\theta d\omega= 1,
   \qquad \text{for every} \quad t \ge 0, \ \Omega \in \R.
\end{equation}
Moreover $\rho$ is $2 \pi$ periodic in $\theta$, continuously depends on $\Omega$ and for every $T>0$ and $\O \in \R$  there exist positive constants $C_{\O,T}$, $\overline{M}$ such that  
\begin{enumerate}
\item  the function $(\o,\theta,\Omega,t) \mapsto \rho(\omega,\theta,\Omega, t)$ has the exponential decay property {\rm ({\bf E})}, with constants $\overline M$ and $C_{\O,T}$;
\item the function $(\o,\theta,\Omega,t) \mapsto \partial_\omega \rho(\omega,\theta,\Omega, t)$ has the exponential decay property {\rm ({\bf E})}, with constants $\overline M$ and $C_{\O,T}/\sqrt{t}$;
\item the function $(\o,\theta,\Omega,t) \mapsto \partial^2_\omega \rho(\omega,\theta,\Omega, t)$ has the exponential decay property {\rm ({\bf E})}, with constants $\overline M$ and $C_{\O,T}/t$. The same assertion holds for the Lie derivative $Y\rho =(\omega\p_\omega -\omega\p_\theta  -\p_t)\rho$.
\end{enumerate}
Furthermore, $\rho$ is the unique positive solution satisfying the properties listed above whenever $g$ has compact support.
\end{theorem}

Following \cite{Sp3}, we look for a solution to the nonlinear Cauchy problem \eqref{eq-Cauchy-pbm} as the fixed point of a map that associates to a given function $\rho$ the solution to a suitable linearized problem, which in our case is relevant to a \emph{degenerate Kolmogorov equation}. With this aim, we introduce the following Cauchy problem
\begin{equation} \label{comune}
\begin{dcases}
  \p^2_{\omega \omega}u + \omega \p_\omega u - \omega \p_\theta u + \Phi_{\tilde{\rho}} (\theta,\Omega,t) \p_\omega u + u - \p_t u = f,\;\;\;\;&(\o,\theta,\O,t)\in\R^3 \times ]T_0,T_1[,\\
  u(\o,\theta,\O,T_0)=u_0(\omega,\theta,\Omega),\;\;&(\omega,\theta,\Omega) \in \R^3,
\end{dcases}
\end{equation}
where
\begin{equation} \label{eq-Phi-comune}
  \Phi_{\tilde{\rho}}(\theta,\Omega,t)=-\Omega - K  \int_\R\int_\R \int_0^{2\pi} g(\Omega')\sin(\theta'-\theta)\tilde{\rho}(\o',\theta',\Omega',t)d\theta'd\omega'd\Omega'.
\end{equation}
If $\tilde{\rho}$ ia s given continuous function satisfying the exponential decay property (see Definition \ref{def-expdecay}), then there exists a \emph{fundamental solution} $\G_{\Omega,\tilde{\rho}}$ useful to represent the solution $u$ to \eqref{comune} as follows
\begin{equation}
	\begin{aligned}
   u(\o,\theta,\Omega,t) = &\int_{\R^2}\G_{\Omega,\tilde{\rho}}(\o,\theta,t;\xi,\eta,T_0) u_0 (\xi,\eta,\Omega)d\eta d\xi-
   \\  -&\int_{T_0}^t \int_{\R^2}    \G_{\Omega, \tilde{\rho}} (\o,\theta,t;\xi,\eta,\tau) f(\xi,\eta,\Omega,\tau)d\eta  d\xi d\tau.
	\end{aligned}
\end{equation} 
when $f,u_0$ verify appropriate assumptions. In Section 2 we study the existence of $\G_{\Omega,\tilde{\rho}}$ and some bounds of $\G_{\Omega,\tilde{\rho}}$ and its derivatives, that provide us with the proof of the existence of the unique solution to \eqref{eq-Cauchy-pbm}.

\medskip

This paper is organized as follows: in Section 2 we recall some results about the fundamental solution and we give some bounds useful in the proof of Theorem \ref{Th-main}, Section 3 is devoted to the proof of Theorem \ref{Th-main}. In Section 4 we define a stable and consistent finite difference scheme that follows the structure of \eqref{ultraparabolica}.


\section{Fundamental Solution}
In this Section we recall some known facts about the partial differential operator
\begin{equation}
	\L_{\Omega,\tilde{\rho}} =	\p^2_{\omega \omega}  + \omega \p_\omega  - \omega \p_\theta  + \Phi_{\tilde{\rho}} (\theta,\O,t) \p_\omega  + 1 - \p_t ,
\end{equation}
appearing in the Cauchy problem \eqref{comune}, where the function $\Phi_{\tilde{\rho}}$ has been introduced in \eqref{eq-Phi-comune}
\begin{equation} 
  \Phi_{\tilde{\rho}}(\theta,\O,t)=-\Omega - K  \int_\R\int_\R \int_0^{2\pi} g(\Omega')\sin(\theta'-\theta)\tilde{\rho}(\o',\theta',\Omega',t)d\theta'd\omega'd\Omega'.
\end{equation}
Note that no derivatives with respect to $\O$ appear in $\L_{\Omega,\tilde{\rho}}$, then, from now on we fix $\O\in \R$ and we consider it as a parameter. 

Let us now recall some known facts about the degenerate differential operator $\L_{\Omega,\tilde{\rho}}$ that will be used in the proof of Theorem \ref{Th-main}. For a reason that will be clear in the following, we also introduce a parameter $\e >0$, and we set
\begin{equation}\label{gammaepsilon}
    \L ^\e := (1 + \e) \p^2_{\o\o} + \omega \p_\omega  - \omega \p_\theta -\p_t.
\end{equation}
This strongly degenerate operator has a fundamental solution $\G^\e= \G^\e (\o, \theta, t; \o_0, \theta_0, t_0)$ which is smooth with respect to the variable $(\o, \theta, t)$ belonging to the set $\R^3 \setminus \{(\o_0, \theta_0, t_0)\}$. We introduce some notation in order to write the explicit expression of $\G^\e$. We first write $\L ^\e$ in the form
\begin{equation}
    \L ^\e := \sum_{i,j=1}^{2} a_{ij}\p^2_{x_ix_j} + \sum_{i,j=1}^{2} b_{ij} x_j \p_{x_i} -\p_t,
\end{equation}
where $x = (x_1, x_2) = (\o, \theta)$, and
\begin{equation}
    A_\e := \left(a_{ij}\right)_{i,j=1,2} = \begin{pmatrix} 1 + \e & 0 \\ 0 & 0 \end{pmatrix}, \quad
    B := \left(b_{ij}\right)_{i,j=1,2} = \begin{pmatrix} 1 & 0 \\ -1& 0 \end{pmatrix}.
\end{equation}
Following H\"ormander (see p. 148 in \cite{H}), we set, for every $t \in \R$, 
\begin{equation}\label{uno}
    E(t) := \exp \big(- t B \big), \qquad
    C_\e (t) := \int_0^t E(s) A_\e E^T(s) \, ds.
\end{equation}
A plain computation shows that
\begin{equation}\label{due}
	E(t)=\begin{pmatrix} \mathrm{e^{-t}}  \;\;\; &0\\ 1-\mathrm{e^{-t}} & 1 \end{pmatrix}, \qquad
	C_\e(s)=\frac{1 + \e}{2} \begin{pmatrix} 1-\mathrm{e}^{-2t} &1-2\mathrm{e}^{-t}+\mathrm{e}^{-2t} \\ 
	1-2\mathrm{e}^{-t}+\mathrm{e}^{-2t} & 2t - 3+4\mathrm{e}^{-t} - \mathrm{e}^{-2t} \end{pmatrix}.
\end{equation} 
Note that the matrix $C_\e(t)$ is symmetric and strictly positive for every $t > 0$, then it is invertible and the fundamental solution $\G^\e$ of $\L^\e u = 0$ is defined for every $(x,t) = (\o, \theta, t) \in \R^2 \times ]0, + \infty[$ as follows
\begin{equation} \label{eq-Gammaepsilon-0}
 \G^\e (x,t) = \frac{1}{4 \pi \sqrt{\text{det} \, C_\e(t)}} 
 \exp \left( - \tfrac{1}{4} \langle C_\e^{-1}(t) x, x \rangle - t \right).
\end{equation}
As customary in the setting of parabolic equation, we set $\G^\e (x,t) = 0$ whenever $t \le 0$. Finally, the fundamental solution with singularity at any point $(x_0,t_0) \in \R^3$ is defined as
\begin{equation} \label{eq-Gammaepsilon}
 \G^\e (x,t; x_0, t_0) := \G^\e(x - E(t-t_0) x_0, t-t_0).
\end{equation}
The following properties of $\G^\e$ will be useful in the sequel
\begin{equation}\label{unieps}
	\int_{\R^2}\G^\e(\o,\theta,t;\x,\y,\t)d\y d\x = 1,\qquad\forall (\o,\theta) \in \R^2, \tau < t.
\end{equation}
Moreover, a direct computation shows that the following identity holds true	
\begin{equation}\label{eq-Gammatilde}
	\int_{\R}\G^\e(\o,\theta,t;\x,\y,\t)d\y = \widetilde \G^\e(\o,t;\x,\t),\qquad\forall (\o,\theta,t, \x, \tau) \in \R^5, \quad \text{with} \ \tau < t.
\end{equation}
Here 
\begin{equation} \label{eq-Gammatilde-0}
 \widetilde \G^\e(\o,t;\x,\t) = \frac{1}{ \sqrt{(1+\e) \pi (1-e^{-2(t-\tau)})}} 
 \exp \left( - \tfrac{| \o - e^{-(t-\tau)}\xi|^2}{(1+\e)  (1-e^{-2(t-\tau)})} - (t-\tau) \right),
\end{equation}
is the fundamental solution to the operator $(1 + \e) \p^2_{\o\o} + \omega \p_\omega -\p_t$.

Let's turn our attention to the Cauchy problem \eqref{comune}
\begin{equation} \label{comune2}
\begin{dcases}
    \L_{\Omega,\tilde{\rho}} \ u = f,\;\;\;\;&(\o,\theta,t)\in\R^2 \times ]T_0,T_1[,\\
    u(\o,\theta,T_0)=u_0 (\o,\theta),\;\;&(\o,\theta) \in \R^2.
\end{dcases}
\end{equation}
The existence of a fundamental solution $\G_{\Omega,\tilde{\rho}}$ to the equation $\L_{\Omega,\tilde{\rho}} \, u = 0$ has been proved by Di Francesco and Pascucci in \cite{DP} via the Levi’s parametrix method. In order to state the main results of \cite{DP}, which include some bounds for $\G_{\Omega,\tilde{\rho}}$ that will be useful in the proof of Theorem \ref{Th-main}, we recall some further notation and known facts on degenerate Kolmogorov operators. 

The expression $(x - E(t-t_0) x_0, t-t_0)$ appearing in \eqref{eq-Gammaepsilon} is related to an invariance property of the differential operator $\L^\e$, which is needed to state the H\"older continuity assumption on the coefficient $\Phi_{\tilde{\rho}}$ appearing in the operator $\L_{\Omega,\tilde{\rho}}$. We first recall that the operator $\L^\e$ is invariant with respect to the change of variable $(\o, \theta, t) \mapsto (\o,\theta,t)\circ (\x,\y,\tau)$, where
\begin{equation}
\begin{aligned}  \label{transl}
  (\o,\theta,t)\circ (\x,\y,\tau) := &((\x,\y) + E(\tau)(\o,\theta) ,t+\tau) \\
  =&(\x+\o\mathrm{e}^{-\t},\y+\theta+\o(1-\mathrm{e}^{-\t}),t+\t),
  \qquad(\o,\theta,t),(\x,\y,\tau)\in \mathbb{R}^3.
\end{aligned}
\end{equation}
The set $\big( \R^{3}, \circ \big)$ is a Lie group with identity $(0,0,0)$, and the inverse of a point $(\o,\theta,t)\in \R^{3}$ is
\begin{equation}
  (\o,\theta,t)^{-1}=(-E(-t)(\o,\theta),-t)=(-\o\mathrm{e}^{t},-\theta-\o(1-\mathrm{e}^{t}),-t).
\end{equation}
Moreover, as usual in the regularity theory for subelliptic operators on Lie groups, an \emph{anisotropic} norm is used to define a quasi-metric. In this case we consider the norm
\begin{equation} \label{seminorm}
  \|(\o,\theta,t)\| = |\o|+|\theta|^{\frac{1}{3}}+|t|^{\frac{1}{2}},\qquad\forall (\o,\theta,t) \in \mathbb{R}^{3},
\end{equation}
and we define the quasi-distance as follows: for every $(\o,\theta,t), (\x,\y,\t) \in  \mathbb{R}^{3}$, we define
\begin{equation}\label{semidistance}
	\begin{aligned}
  d((\o,\theta,t),(\x,\y,\t))&=\|(\x,\y,\t)^{-1}\circ (\o,\theta,t)\|\\
  &=|\o-\x\mathrm{e}^{\t-t}|+|\theta-\y+\x(\mathrm{e}^{\t-t}-1)|^{\frac{1}{3}}+|t-\t|^{\frac{1}{2}}.
	\end{aligned}
\end{equation}
We are now in position to give the following
\begin{definition}[H\"older continuous function]\label{hold}
Let $\alpha \in ] 0,1]$, and let $U$ be an open subset of $\mathbb{R}^{3}$. We say that a function $f:U \longrightarrow \mathbb{R}$ is H\"older continuous with exponent $\alpha$ in $U$ (in short  $f \in C^\alpha(U)$) if there exists a positive constant $k$ such that
\begin{equation}
  |f(\o,\theta,t)-f(\x,\y,\t)|\leq k \, d((\o,\theta,t),(\x,\y,\t))^\alpha,\qquad 	\forall (\o,\theta,t),(\x,\y,\t) \in U.
\end{equation}
\end{definition}

We next recall Theorem 1.4 in \cite{DP}, which provides us with an existence result for a fundamental solution $\G_{\Omega,\tilde{\rho}}$ to $\L_{\Omega,\tilde{\rho}} u = 0$.

\begin{theorem}\label{cauchylow}
Let suppose that the  coefficient $\Phi_{\tilde{\rho}} (\o,\theta,t)$ of $\L_{\O, \tilde{\rho}}$  is bounded and such that there exists $\a\in ]0,1]$ and a positive constant $k$ such that
\begin{equation}
	|\Phi_{\tilde{\rho}} (\theta,\O,t)-\Phi_{\tilde{\rho}} (\tilde{\theta},\Omega,t)|\leq k \ d((\theta,\O,t),(\tilde{\theta},\O,t))^\alpha,\qquad 	\forall \theta,	\tilde{\theta} 	\in \R, (\O,t) \in \R^2.
\end{equation}  
Then there exists a fundamental solution $\G_{\Omega,\tilde{\rho}}$ to $\L_{\Omega,\tilde{\rho}}$ with the following properties:
\begin{enumerate}
  \item $\G_{\Omega,\tilde{\rho}}(\cdot;\x,\y,\t) \in \mathrm{L}_{\loc}^1(\R^{3} ) \cap C(\R^{3}\setminus \{(\x,\y,\t)\})$ for every $(\x,\y,\t) \in \R^{3}$;
  \item $\G_{\Omega,\tilde{\rho}}(\cdot;\x,\y,\t)$ is a classical solution to $\L_{\Omega,\tilde{\rho}} \ u =0$ in $\R^{3}\setminus \{(\x,\y,\t)\}$ for every $(\x,\y,\t) \in \R^{3}$;
  \item the following identity holds
  \begin{equation}\label{unitary}
  \int_{\R^2}\G_{\Omega,\tilde{\rho}}(\o,\theta,t;\x,\y,\t)d\y d\x = 
  \mathrm{e}^{-(t-\tau)},\qquad\forall (\o,\theta) \in \R^2, \tau < t.
  \end{equation}
  \item\label{sei} for every $\e>0$ and $T>0$  there exists a constant $\overline{C} = \overline{C} \big( \e, T,  \| \Phi_{\tilde{\rho}} \|_{L^\infty(\R^3)}\big) $, 
  such that 
\begin{equation}
	\begin{aligned}\label{estimatesderivate}
 0< \;\G_{\Omega,\tilde{\rho}}(\o,\theta,t;\x,\y,\t)&\leq \overline{C} \G^\e(\o,\theta,t;\x,\y,\t),
  \\|  \p\G_{\Omega,\tilde{\rho}}(\o,\theta,t;\x,\y,\t)| &\leq \frac{\overline{C}}{\sqrt{(t-\tau)^i}} \G^\e(\o,\theta,t;\x,\y,\t),
  \end{aligned} 
\end{equation}
where $\G^\e$ is the function defined in \eqref{eq-Gammaepsilon} $\p$ stands for $\p_{\o}$ ($i=1$), $\p^2_{\o\o}$ or Lie derivative $Y$ ($i=2$), for every  $(\o,\theta,t),(\x,\y,\t) \in \R^{3}$ with $0<t-\tau<T$.  
\end{enumerate}
Let's consider the following Cauchy problem
\begin{equation} \label{comune3}
	\begin{dcases}
 \L_{\Omega,\tilde{\rho}}u(\o,\theta,t) = f(\o,\theta,t),\qquad&(\o,\theta,t)\in\R^2 \times ]T_0,T_1[,\\
		u(\o,\theta,T_0)=u_0(\o,\theta),\qquad&(\o,\theta) \in \R^2,
	\end{dcases}
\end{equation}
where $u_0 \in C(\R^2)$ is such that
  \begin{equation}\label{hpg}
    |u_0(\o,\theta)|\leq  C_1\mathrm{e}^{C_1(\omega^2+\theta^2)},\qquad\forall (\o,\theta) \in \R^2,	
  \end{equation}
and $f \in C(\R^2\times (T_0,T_1))$ is such that
  \begin{equation}
    |f(\o,\theta,t)|\leq  \frac{C_1\mathrm{e}^{C_1(\omega^2+\theta^2)}}{\sqrt{t}},\qquad\forall (\o,\theta,t) \in \R^2\times (T_0,T_1),
  \end{equation}
for some positive constant $C_1$, and for any compact subset $M$ of $\R^2$ there exists a positive constant $C$, $\beta \in ]0,1[$ and $\delta >0$ such that
  \begin{equation}\label{est}
  |f(\o,\theta,t)-f(\x,\y,t)|\leq C \frac{d((\o,\theta,t),(\x,\y,t))^\beta}{t^{1-\delta}},\qquad\forall (\o,\theta),(\x,\y) \in M, t \in ]T_0,T_1[.	
  \end{equation}
Then there exists $T \in ]T_0,T_1[$, only depending on growth constant $C_1$, such that the function
  \begin{equation} \label{conv}
  	\begin{aligned}
    &u(\o,\theta,t)=\\
    &\int_{\R^2}\G_{\Omega,\tilde{\rho}}(\o,\theta,t;\x,\y,T_0)u_0(\x,\y)d\y d\x - \int_{T_0}^t\int_{\R^2}\G_{\Omega,\tilde{\rho}}(\o,\theta,t;\x,\y,\t)f(\x,\y,\tau)d\y d\x d\tau,
	\end{aligned}
  \end{equation}
is solution to the previous Cauchy problem in $ ]T_0,T[$.

Moreover if $u$ is a solution to the Cauchy problem with null $f$ and $g$, and verifies the following estimate 
  \begin{equation}
	|u(\o,\theta,t)|\leq  C\mathrm{e}^{C(\o^2+\theta^2)},\qquad\forall (\o,\theta,t) \in \R^2\times ]T_0,T_1[,	
\end{equation} 
for some $C>0$, then $u \equiv  0$.
\end{theorem}

\section{Proof of Theorem \ref{Th-main}.}
In this Section we prove existence, uniqueness and regularity of classical solutions to the Cauchy problem \eqref{eq-Cauchy-pbm} in $S_T = \R^3 \times ]0,T[$, which we recall here for reader's convenience
\begin{equation}
	\begin{dcases} 
 \L_{\O,\r}\r=\frac{\p^2\rho}{\p\omega^2} + \frac{\p}{\p\omega} \left[(\omega + \Phi_\rho(\theta,\O,t))\rho\right] - \omega \frac{\p \rho}{\p\theta} -\frac{\p\rho}{\p t} =0,\qquad &\mathrm{in} \ S_T,\\
 \rho(\o,\theta,\Omega,0)=\rho_0(\o,\theta,\Omega),\qquad &\mathrm{in}	\ \R^3,
	\end{dcases}
\end{equation} 
where 
\begin{equation}
\Phi_\rho(\theta,\O,t) = - \O - K \int_\R \int_\R \int_0^{2\pi} 
g(\Omega')\sin(\theta'-\theta)\rho(\o',\theta',\Omega',t)d\theta'd\omega'd\Omega'.
\end{equation}

As said in the Introduction, we find a solution to the nonlinear problem \eqref{eq-Cauchy-pbm} as the fixed point of a map that associates to a given function $\rho$ the solution to the linearized problem \eqref{comune}, which is obtained by using the fundamental solution $\G_{\Omega,\tilde{\rho}}$. We list the assumptions about the initial datum $\rho_0$ we will adopt in this iterative argument.
\begin{enumerate}
  \item $\rho_0 \in C(\R^3)$;
  \item $\rho_0$ is strictly positive and verifies {\rm ({\bf E})} in Definition \ref{def-expdecay};
  \item $\rho_0$ is $2\pi$ periodic with respect to $\theta$;
  \item for every $\Omega \in \R$ we have \label{normal}
  \begin{equation}
    \int_{]0, 2\pi[ \times \R}\rho_0(\o,\theta,\Omega) d\theta d\omega= 1.
  \end{equation}
\end{enumerate} 
Moreover we recall that the natural frequency distributions of oscillators $g(\O)$ is a probability density, 
satisfying \eqref{hp-g}. 
We first prove a preliminary result. Here and in the sequel, we set $S_T^*:= \R^2 \times ]0,T[$, and we fix $\O \in \R$. Moreover, we consider the function $\G^\e$, defined in \eqref{eq-Gammaepsilon}, for a fixed $\e >0$, as we will rely on \eqref{estimatesderivate}.

\begin{lemma}\label{lemmam}
Consider for $T>0$ the Cauchy problem \eqref{comune} in $S_T$ with $f=0$. If $\tilde\rho$ is continuous and verifies property  {\rm ({\bf E})}, then the function $u$ defined as
\begin{equation}\label{repre}
  u(\o,\theta,\Omega,t) = \int_{\R^2}\G_{\Omega,{{\tilde\rho}}}(\o,\theta,t;\xi,\eta,0)
  \rho_0(\xi,\eta,\Omega)d\eta d\xi,
\end{equation}
is a classical solution $u$ to \eqref{comune} with $T_0 = 0, T_1 = T$, and $u_0 = \varrho_0$ . Moreover
\begin{enumerate}
  \item $u$ verifies property  {\rm ({\bf E})} with constants $\overline M, C_{\Omega,T}$. In addition we have
  \begin{equation}\label{stimay} 
  | \p u(\o,\theta, \Omega,t)| \leq \frac{C_{\Omega,T}}{\sqrt{t^i}} \mathrm{e}^{-\overline M \omega^2}, \quad \forall (\omega, \theta, t) \in S_T^*,
   \end{equation}
where $\p$ stands for $\p_\omega$ (and \eqref{stimay} holds with $i=1$), $\p^2_{\omega \omega}$ or Lie derivative $Y$ (and \eqref{stimay} holds with $i=2$); 
  \item $u$ is $2 \pi$ periodic with respect to $\theta$;
  \item  for every $\Omega \in \R$ and for every $t \in [0,T[$ we have
  \begin{equation}
 \int_{]0, 2\pi[ \times \R} u(\o,\theta,\Omega,t) d\theta d\omega= 1.
  \end{equation}
\item $u(\o,\theta,\Omega,t)>0$ for every $t\geq 0$.
\end{enumerate}
\end{lemma}
\begin{proof} 
As a direct consequence of the property {\rm ({\bf E})}, we can easily see that the function $\Phi_{\tilde\rho}$ defined in \eqref{eq-Phi-comune} is bounded. Moreover, if we consider $\Phi_{\tilde\rho}$ as a function of the variable $(\omega, \theta,\O,t)$ by letting $\Phi_{\tilde\rho}(\omega, \theta,\O,t) := \Phi_{\tilde\rho}(\theta,\O,t)$, there exists a positive constant $L$ such that 
\begin{equation} \label{eq-Phi-Lip}
  |\Phi_{\tilde\rho}(\omega_1, \theta_1,\O,t) - \Phi_{\tilde\rho}(\omega_2, \theta_2,\Omega,t)| 
  \leq L\; \mathrm{d}((\omega_1, \theta_1,\O,t),(\omega_2, \theta_2,\O,t)),
\end{equation} 
for every $(\omega_1, \theta_1,\O,t), (\omega_2, \theta_2,\O,t) \in S_T$. Therefore the function $u$ defined in \eqref{repre} 
is a classical solution $u$ to \eqref{comune}, by Theorem \ref{cauchylow}. We next prove that $u$ has the properties listed in the statement of the Lemma.

\medskip
In order to prove that $u$ and all the derivatives $\p u$ listed in the point \emph{(.1)} verify the property {\rm ({\bf E})}, we first observe that the identity \eqref{repre} can be differentiated under the integral sign, \emph{i.e.}
\begin{equation}\label{derivatero}
  \p u(\o,\theta,\Omega,t) = 
  \int_{\R^2} \p \G_{\Omega,{{\tilde\rho}}}(\o,\theta,t;\xi,\eta,0)\rho_0(\xi,\eta,\Omega)d\eta d\xi.
\end{equation}
Then by  point \emph{(.4)} of Theorem \ref{cauchylow} we have
\begin{equation}
 | \p u(\o,\theta,\Omega,t)| \leq \frac{\overline{C}_{\O,T}}{\sqrt{t^i}} \int_{\R^2} \G^\e(\o,\theta,t;\xi,\eta,0)\rho_0(\xi,\eta,\Omega)d\eta d\xi,
\end{equation}
for some $\overline{C}_{\O,T}>0$ that depend on  $\e, T,  \| \Phi_{\tilde{\rho}\|_{L^\infty(\R^3)}}$, and therefore on $\O$. 	

Property {\rm ({\bf E})} follows from \eqref{estimatesderivate} and from the exponential  decay of the initial datum. Consider a given $\omega \in \R$, by \eqref{unieps} we have
\begin{equation}\label{firstupper}
|\p u(\o,\theta, \Omega,t)| \leq \frac{\overline{C}_{\O,T}C}{\sqrt{t^i}}\int_{\R^2}
\G^\e(\o,\theta,t;\x,\y,0)\mathrm{e}^{-M\x^2}d\x d\y .
\end{equation}
If $|\omega| > 1$, we first recall the equation \eqref{eq-Gammatilde-0} where $\widetilde{\G}^\e$ has been defined, then we split the above integral in two parts
\begin{equation}
\begin{split}
    \int_{\R^2} & \G^\e(\o,\theta,t;\x,\y,0)\mathrm{e}^{-M\x^2}d\x d\y = 
    \int_{- \infty}^{+ \infty} \widetilde{\G}^\e (\o,t;\x,0) \mathrm{e}^{-M\x^2} d\x \\
    & \qquad = \int_{\big\{|\x| \geq \frac{|\o |}{2}\big\}} \widetilde{\G}^\e (\o,t;\x,0) \mathrm{e}^{-M\x^2} d\x 
    + \int_{\big\{|\x|<\frac{|\o|}{2}\big\}} \widetilde{\G}^\e (\o,t;\x,0) \mathrm{e}^{-M\x^2} d\x.
\end{split}
\end{equation}
Concerning the first integral, we immediately have 
\begin{equation}\label{firstupper-1}
 \int_{\big\{|\x| \geq \frac{|\o |}{2}\big\}} \widetilde{\G}^\e (\o,t;\x,0) \mathrm{e}^{-M\x^2}d\x \le 
 \mathrm{e}^{-M\frac{\o^2}{4}} \int_{- \infty}^{+ \infty} \widetilde{\G}^\e (\o,t;\x,0) d\x 
 = \mathrm{e}^{-M\frac{\o^2}{4}}.
\end{equation}
We use the change of variable $x=\frac{\mathrm{e}^{-t}\x-\o}{\sqrt{(1+\e)(1-e^{-2t})}}$ in the second integral. Note that, if $|\x|<\frac{|\o|}{2}$, then $|x|>\frac{|\o|}{2 \sqrt{(1+\e)(1-e^{-2t})}}$, so that
\begin{equation}
\begin{split}
    \int_{\big\{|\x|<\frac{|\o|}{2}\big\}} & \widetilde{\G}^\e (\o,t;\x,0) \mathrm{e}^{-M\x^2} d\x \le \\
	\le & \int_{\big\{|x|>\frac{|\o|}{2 \sqrt{(1+\e)(1-e^{-2t})}}\big\}} \frac{\mathrm{e}^{-x^2}}{\sqrt{\pi}}
	\mathrm{e}^{-M\mathrm{e}^{2t}\left(x\sqrt{M(1+\e)(1-e^{-2t})}+\o\right)^2}dx \le \\
	\le & \tfrac{\mathrm{e}^{-\frac{\o^2}{4(1+\e)(1-e^{-2t})}}}{M\sqrt{ (1+\e)(1-e^{-2t})}} <
	\tfrac{\mathrm{e}^{-\frac{1}{8(1+\e)(1-e^{-2t})}}}{M\sqrt{ (1+\e)(1-e^{-2t})}} \cdot
	\mathrm{e}^{-\frac{\o^2}{8(1+\e)}} \le \tilde C_{\e,T} \cdot \mathrm{e}^{-\frac{\o^2}{8(1+\e)}},
\end{split}
\end{equation}
where we have used the assumption that $|\omega| > 1$ and the fact that the function appearing in the last line of the above display is bounded in its domain $\R^+$. Then \eqref{stimay} in the case $|\omega| > 1$ follows from the above inequality and \eqref{firstupper-1}, if we set
\begin{equation} \label{eq-COmegaT}
 \overline M = \min \big\{ \tfrac{M}{4}, \tfrac {1}{8(1+\e)} \big\}.
\end{equation}
If $|\omega| \le 1$ by \eqref{firstupper} we have
\begin{equation}
	|\p u(\o,\theta, \Omega,t)| \leq \frac{\overline{C}_{\O,T} \, C \, \mathrm{e}^{\overline{M}}}{\sqrt{t^i}}\mathrm{e}^{-\overline{M}\o^2},
\end{equation} 
then the thesis follows if we set
\begin{equation}\label{expr-c}
	 C_{\O,T} :=  \overline{C}_{\O,T} \, C  \max \{2, 2 \, \tilde C_{\e,T},\mathrm{e}^{\overline{M}}\}.
\end{equation}
The proof of the periodicity stated in \emph{(.2)} follows from a standard uniqueness argument. We omit the details.

In order to prove the assertion \emph{(.3)} of the Lemma, we first notice that the continuity of the functions $\p_\o u$ and $Y u = \o\p_\o u -\o\p_\theta u -\p_t u$ implies the continuity of the Lie derivative $\tilde Y  u := Y u - \o \p_\o u = -\o\p_\theta u -\p_t u$. Let us define for every $R>0$ and $0<t_0<t_1$ the sets
\begin{equation}
	\begin{aligned}
		&C_R=\left\{(\o,\theta,t)\in \R^3:-R\leq\o\leq R,\;0\leq \theta \leq 2\pi, \;t_0\leq t\leq t_1\right\},\\
		&B_R=\left\{(\o,\theta)\in \R^2:-R\leq\o\leq R,\;0\leq \theta \leq 2\pi \right\}.\\
	\end{aligned}
\end{equation}
If we integrate in $C_R$ the equation in \eqref{comune}  we obtain
\begin{equation}
  \int_{C_R}-\tilde Y  u  d\omega d\theta dt = \int_{C_R} \p_{\omega}(\p_{\omega}u + (\omega 
  + \Phi_{\tilde\rho}(\theta,\Omega,t))u) d\omega d\theta dt.
\end{equation}
Then by the divergence theorem we have
\begin{equation}
	\begin{aligned}
	&\int_{B_R}u(\o,\theta,\O,t_1)d\o d\theta-\int_{B_R}u(\o,\theta,\O,t_0)d\o d\theta+\int_{t_0}^{t_1}\int_{-R}^{R} \o u(\o,0,\O,t)d\o dt-\\ \\
	&-\int_{t_0}^{t_1}\int_{-R}^{R} \o u(\o,2\pi,\O,t)d\o dt=\int_{C_R} \p_{\omega}(\p_{\omega}u + (\omega 
	+ \Phi_{\tilde\rho}(\theta,\Omega,t) ) u) d\omega d\theta dt.
	\end{aligned}
\end{equation}
The last two integrals in the left-hand side sum to zero by periodicity, moreover by the exponential decay in $\omega$ we have
\begin{equation}
	\begin{aligned}
		&\lim_{R\to \infty}\int_{C_R} \p_{\omega}(\p_{\omega}u 
		+ (\omega + \Phi_{\tilde\rho}(\theta,\Omega,t))u) d\omega d\theta dt=0.
	\end{aligned}
\end{equation}
Then for every $0 \le t_0<t_1$ we have
\begin{equation}
	\int_{]0,2\pi[\times \R}u(\o,\theta,\O,t_0)d\o d\theta=\int_{]0,2\pi[\times \R}u(\o,\theta,\O,t_1)d\o d\theta.
\end{equation}
The thesis then follows from the property of $\rho_0$ by choosing $t_0=0$.

The proof of the point \emph{(.4)} is a consequence of the representation formula \eqref{repre}, since both $\G_{\Omega,{\tilde\rho}}$ and $\rho_0$ are positive functions.
\end{proof}

\medskip
We are now ready to prove an existence result for the Cauchy problem \eqref{eq-Cauchy-pbm}. 

\begin{proposition}\label{result}
For every $\Omega, T \in \R$, with $T>0$ and $\rho_0$ satisfying the property {\rm ({\bf E})}, there exists a classical solution $\rho = \rho(\o,\theta,\Omega,t)$ to \eqref{eq-Cauchy-pbm} that verifies property {\rm ({\bf E})} in the set $S_T^*$. Moreover $\rho$ is unique and can be represented by the fundamental solution $\G_{\Omega,{{\rho}}}$ of $\L_{\Omega,{{\rho}}}$ as follows
\begin{equation}\label{eq-repr-rho}
  \rho(\o,\theta,\Omega,t) = \int_{\R^2}\G_{\Omega,{{\rho}}}(\o,\theta,t;\xi,\eta,0)
  \rho_0(\xi,\eta,\Omega)d\eta d\xi.
\end{equation}
\end{proposition}
\begin{proof}
We define a sequence of function $\big\{ \rho_n \big\}_{n \in \N \cup \{ 0\}}$ as follows. We set $\rho_0(\o,\theta,\Omega,t)=\rho_0(\o,\theta,\Omega)$ for every $(\o,\theta,t) \in S_T^*$. Then, for every $n \in \mathbb{N}$, we let 
\begin{equation}
  \Phi_{\rho_{n-1}}(\theta,\O,t)= - \O - K \int_\R \int_\R \int_0^{2\pi} 
  g(\Omega')\sin(\theta'-\theta)\rho_{n-1}(\o',\theta',\Omega',t)d\theta'd\omega'd\Omega',
\end{equation}
and we define $\rho_n$ as the solution to
\begin{equation}
	\begin{dcases} \label{probappr}
    \L_{\O,\r_{n-1}}\r_n=\p^2_{\omega\omega}\rho_n + \p_\omega [(\omega + \Phi_{\rho_{n-1}}(\theta,\O,t))\rho_n] - \omega \p_\theta \rho_n -\p_t \rho_n=0, \\
    \rho_n (\o,\theta,\Omega,0)=\rho_0(\o,\theta,\Omega).
	\end{dcases}
\end{equation}
The solution to the non-linear Cauchy problem will be defined as the limit of the sequence $\big\{ \rho_n \big\}_{n \in \N \cup \{ 0\}}$. We rely on a compactness argument to prove the convergence of a subsequence of $\big\{ \rho_n \big\}_{n \in \N \cup \{ 0\}}$. However, due the presence of the term $\rho_{n-1}$ in the coefficient $\Phi_{\rho_{n-1}}$ in \eqref{probappr}, we need to show that the whole sequence converges.

\medskip

We now prove that the sequence $\big\{ \rho_n \big\}_{n \in \N \cup \{ 0\}}$ has the property {\rm ({\bf E})}, with constants that does not depend on $n$. As $\rho_0 \in C(\R^3)$ and verifies property {\rm ({\bf E})}, we can apply Lemma \ref{lemmam}, therefore for every $n \in \N$ there exists a classical solution $\r_n$ in $\R^2\times [0,T[$ such that:
\begin{itemize}
  \item $\rho_n$ is $2 \pi$ periodic in $\theta$; 
  \item $\rho_n$ is strictly positive for every $t\geq 0$;
  \item $\rho_n$ is normalized 
  \begin{equation}\label{rho-normalized}
  \int_{]0,2\pi[\times\R} \rho_n(\o,\theta,\Omega,t) d\theta d\omega = 1,\qquad \forall\; t \in [0,T[.
  \end{equation}
\end{itemize}

We recall that $\G_{\Omega,\r_{n-1}}$ denotes the fundamental solution of the operator $\L_{\O,\r_{n-1}}$ appearing in \eqref{probappr}. The assertion \emph{(.4)} of Theorem \ref{cauchylow} yields
\begin{equation}
	\begin{aligned}\label{stimechemiservono}
   \G_{\Omega,\r_{n-1}}(\o,\theta,t;\xi,\eta,0)&\leq \overline{C}_{\O,T} \G^\e(\o,\theta,t;\xi,\eta,0),\\ 
  |  \p\G_{\Omega,\r_{n-1}}(\o,\theta,t;\xi,\eta,0)| &\leq \frac{\overline{C}_{\O,T}}{\sqrt{t^i}} \G^\e(\o,\theta,t;\xi,\eta,0),
	\end{aligned} 
\end{equation}
where the constant $\overline{C}_{\O,T}>0$ only depends on $\e, T$, and on $\| \Phi_{{\rho_{n-1}}} \|_{L^\infty(S_T^*)}$. Moreover, also using \eqref{eq-Phi-Lip} and \eqref{rho-normalized}, we see that
\begin{equation} \label{eq-bound-Phi}
\begin{split}
 \| \Phi_{{\rho_{n-1}}} \|_{L^\infty(S_T^*)} &  \le |\Omega| + K, \\
 |\Phi_{{\rho_{n-1}}} (\o_1,\theta_1,\Omega,t)- \Phi_{{\rho_{n-1}}}(\o_2,\theta_2,\Omega,t)| &  \leq 
 K \, \mathrm{d}((\o_1,\theta_1,\Omega,t),(\o_2,\theta_2,\Omega,t)),
\end{split}
\end{equation}
for every $(\o_1,\theta_1,\Omega,t),(\o_2,\theta_2,\Omega,t) \in S_T$, and for every $n \in \N$. Note that, in particular, the constant $\overline{C}_{\O,T}$ in \eqref{stimechemiservono} does not depend on $n$. In the sequel, we will use the fact that, from the boundedness and the Lipschitz continuity of the function $\Phi_{{\rho_{n-1}}}$, the $\beta$-H\"older continuity of $\Phi_{{\rho_{n-1}}}$ directly follows for every $\beta \in ]0,1[$. 

As a consequence of Lemma \ref{lemmam}, the sequence $\{\rho_n\}_{n \in \mathbb{N}}$ has an uniform exponential decay in $\omega$, that is
\begin{equation}\label{rhosup}
 0 < \rho_n(\o,\theta,t) \leq  C_{\O,T} \, \mathrm{e}^{- \overline M \omega^2}, 
\qquad \forall (\o,\theta,t) \in \R^2 \times ]0,T[, \quad n \in \N.
\end{equation}
Moreover, $\p^2_\omega \rho_n$, and $Y \rho_n$ have an uniform exponential decay in $\omega$ as well, with a constant  $C_{\O,T}/{t}$. 
Thus, the sequence $\{\rho_n\}_{n \in \mathbb{N}}$ is equi-bounded and equi-continuous on every compact subset of $\R^2 \times ]0,T]$. Hence, there exists a subsequence $\{\rho_{n_k}\}_{k \in \mathbb{N}}$ and a function $\rho \in C(S^*_T)$ such that 
\begin{equation}
 \rho_{n_k} \rightarrow \rho, \quad \text{as} \quad {k \to + \infty},
\end{equation}
uniformly on every compact subset of $\R^2 \times ]0,T]$.

As said above, we need to show that the whole sequence $\{\rho_n\}_{n \in \mathbb{N}}$ converges to $\rho$. We will prove a stronger result, namely 
\begin{equation} \label{eq-rho_n-to-rho}
 \| \rho_{n} - \rho \|_{L^\infty (S_T^*)} \rightarrow 0, 
 \quad \text{as} \quad {n \to + \infty},
\end{equation}
where ${S_T^*} = \R^2 \times ]0,T[$. Aiming at proving \eqref{eq-rho_n-to-rho}, we preliminarily note that the constant $\overline{C}_{\O,T}$ in \eqref{stimechemiservono} depends on the quantities $\O$ and $T$ as follows
\begin{equation}\label{explicit-c}
	\overline{C}_{\O,T} = C_\e+\frac{\sqrt{T}}{2}\sum_{j=1}^{+\infty} (|\O|+K)^j(1+\sqrt{T})^j\frac{\sqrt{\pi}}{\G_E(\frac{j}{2})},
\end{equation}
where $C_\e>0$  depends only on $\e$, which is fixed, and $\G_E$ denotes the Euler Gamma function. We obtain the above expression by repeating the computations made in \cite{DP} in the particular case of the operator $\L$.
Then, by recalling the definition \eqref{expr-c} of $C_{\O,T}$, the hypothesis \eqref{hp-g} yields
\begin{equation}\label{G}
	G:=\int_{\R} g(\O){C}_{\O,T} \, d\O < +\infty.
\end{equation} 
Let us now fix a positive $t_0$, which needs to be chosen small enough to have  
\begin{equation}
  q := 2 K \pi^{\frac{3}{2}}\overline{C}_{\O,T}G\sqrt{ \tfrac{t_0}{A}} < 1,
\end{equation}
where we have set
\begin{equation}
	 A := \frac{\overline{M}}{1+2(1+\e)\overline{M}(\mathrm{e}^{2T}-1)}.
\end{equation}
Here $\overline M$ is the constants in the exponential decay. Then, we define $\kk=\lfloor T/t_0\rfloor +1$ and we set $T_k=k t_0$ for $k=0,...,\kk-1$, and $T_\kk = T$. For every $n \in \N$ we consider the function 
\begin{equation}
v_n(\o,\theta,\Omega ,t) := \rho_{n+1} (\o,\theta,\Omega, t) - \rho_n (\omega,\theta, \Omega, t).
\end{equation}
It is the unique bounded solution to the following Cauchy problem
\begin{equation}\label{cahu}
  \begin{cases} 
   \L_{\O, \rho_{n}} v_n = (\Phi_{\rho_{n}}-\Phi_{\rho_{n-1}})\p_\omega \rho_n,\quad
   \mathrm{in}\; \R^2 \times ]T_k,T_{k+1}[, \\
   v_n(\o,\theta,\Omega, T_k) = \rho_{n+1} (\o,\theta,\Omega, T_k) - \rho_n (\o,\theta,\Omega, T_k).
  \end{cases}
\end{equation}
Now we represent the function $v_n$ by using the fundamental solution $\G_{\O, \rho_{n}}$ of $\L_{\O, \rho_{n}}$. The above PDE can be written as $\L_{\O, \rho_{n}} v_n = f_n$, with 
\begin{equation}\label{nonul}
   f_n (\o,\theta,\Omega,t) := (\Phi_{\rho_{n}}(\theta,\Omega,t)-
   \Phi_{\rho_{n-1}}(\theta,\Omega,t))\p_\omega \rho_n(\o,\theta,\Omega,t).
\end{equation} 
In order to apply Theorem \ref{cauchylow}, we claim that there exist three positive constants $C_1, \gamma$ and $\beta$, with $\beta, \gamma <1$, such that
\begin{equation} \label{eq-bounds-fn}
\begin{split}
  |f_n(\o,\theta,\Omega,t)| & \leq \frac{C_1}{\sqrt{t}}, \\
  |f_n(\o_1,\theta_1,\Omega,t)-f_n(\o_2,\theta_2,\Omega,t)| & \leq 
 \frac{C_1}{t^{1-\gamma}}\; \mathrm{d}((\o_1,\theta_1,\Omega,t),(\o_2,\theta_2,\Omega,t))^\beta ,
\end{split}
 \end{equation}
for every $(\o,\theta,\Omega,t) (\o_1,\theta_1,\Omega,t),(\o_2,\theta_2,\Omega,t) \in S_T$, and for every $n \in \N$.
The first inequality in \eqref{eq-bounds-fn} directly follows from the first inequality in \eqref{eq-bound-Phi} and from the uniform exponential decay of $\p_\omega \rho_n$. In order to prove the second inequality, we claim that there exist two positive constants $C$ and $\beta$, with $\beta <1$, such that
\begin{equation} \label{eq-secondbountfn}
	  |\p_\o\rho_n(\o_1,\theta_1,\Omega,t)-\p_\o\rho_n(\o_2,\theta_2,\Omega,t)| \leq 
	\frac{C}{t^{\frac{1+\beta}{2}}}\; \mathrm{d}((\o_1,\theta_1,\Omega,t),(\o_2,\theta_2,\Omega,t))^\beta,
\end{equation}
for every $ (\o_1,\theta_1,\Omega,t),(\o_2,\theta_2,\Omega,t) \in S_T$, and for every $n \in \N$. 
The conclusion of the proof of \eqref{eq-bound-Phi} plainly follows from the exponential decay of $\p_\omega \rho_n$ combined with \eqref{eq-bound-Phi} and \eqref{eq-secondbountfn}. 

The estimate \eqref{eq-secondbountfn} is a direct consequence of the identity
\begin{equation}
	\p \rho_n(\o,\theta,\Omega,t) = 
	\int_{\R^2} \p_\o \G_{\Omega,{{\rho_n}}}(\o,\theta,t;\xi,\eta,0)\rho_0(\xi,\eta,\Omega)d\eta d\xi.
\end{equation}
Then we use the following property the fundamental solution, proved in Theorem 3.2 of \cite{DPa}: for every $\beta \in ]0,1[$ there exists a positive constant $c_\beta$ such that
\begin{equation}\begin{aligned}
\big|\p_\omega \G_{\O, \rho_{n}}(\o_1,\theta_1,t;\x,\eta, 0)- 	
\p_\omega \G_{\O, \rho_{n}} & (\o_2,\theta_2,t;\x,\eta, 0) \big| \leq\\& 
c_\beta \frac{\mathrm{d}((\o_1,\theta_1,t),(\omega_2,\theta_2,t))^\beta}{t^{\frac{1+\beta}{2}}}
\G^\e(\o_1,\theta_1,t;\x,\eta, 0),
\end{aligned}\end{equation}
for every $t \in ]0,T[, (\o_1,\theta_1),(\o_2,\theta_2),(\x,\y)\in \R^2$, being $\G^\e$ the function defined in \eqref{gammaepsilon}. Finally, by using the exponential decay of the initial datum and \eqref{unieps} we find
\begin{equation}\begin{aligned}
	 \big| \p_\o\rho_n(\o_1,\theta_1,\Omega,t)-\p_\o\rho_n (\o_2,\theta_2,\Omega,t) \big| \leq & \\ c_\beta \frac{\mathrm{d}((\o_1,\theta_1,t),(\omega_2,\theta_2,t))^\beta}{t^{\frac{1+\beta}{2}}} & 
	 \int_{\R^2}\G^\e(\o_1,\theta_1,t;\x,\eta, 0)\rho_0(\x,\y)d\x d\y\leq  \\ 
	 & \qquad \qquad \qquad \qquad \overline c_\beta 
	 \frac{\mathrm{d}((\o_1,\theta_1,t),(\omega_2,\theta_2,t))^\beta}{t^{\frac{1+\beta}{2}}},
\end{aligned}\end{equation}
for some $\overline{c}_\b>0$. This concludes the proof of \eqref{eq-bounds-fn} with $\beta$ arbitrarily chosen in $]0,1[$ and $\gamma = \frac{1-\beta}{2}$.

Once the bounds \eqref{eq-bounds-fn} have been proved, we represent $v_n$ as follows
\begin{equation}
\begin{aligned}\label{represent}
 v_n(\o,\theta,\Omega ,t) = 
 \int_{\R^2} \G_{\Omega,{{\rho_n}}} & (\o,\theta,t;\xi,\eta,T_k) v_n ( \x,\eta, \Omega, T_k) \, d\eta d\xi \\
 & -\int_{T_k}^t \int_{\R^2} \G_{\Omega,{{\rho_n}}} (\o,\theta,t;\x,\eta,\tau)
 f_{n}(\xi, \eta, \O, \tau) \, d\eta  d\xi d\tau.
\end{aligned}
\end{equation}

In order to simplify the remaining part of the proof, we define recursively the constants 
\begin{equation}
	M_1:= \frac{\overline{M}}{1+\overline{M}(1+\e)(\mathrm{e}^{2T_1}-1)}, \quad M_{k+1}:= \frac{M_k}{1+M_k(1+\e)(\mathrm{e}^{2(T_{k+1}-T_k)}-1)},\quad k=1,...,\kk-1,
\end{equation}
and we note that, by a plain induction argument, we find
\begin{equation}
M_{k}=\frac{\overline{M}}{1+\overline{M} (1+\e)k(\mathrm{e}^{2t_0}-1)},\quad M_{\hat{k}}=\frac{\overline{M}}{1+\overline{M} (1+\e)((\hat{k}-1)(\mathrm{e}^{2t_0}-1)+(\mathrm{e}^{2(T-T_{\hat{k}-1})}-1))}.
\end{equation}
In particular, the inequalities below hold
\begin{equation}\label{eq-mk}
	\overline{M}>M_k>M_{k+1} \ge A, \qquad k=1,...,\kk-1,
\end{equation}  
since
\begin{equation}
(\hat{k}-1)(\mathrm{e}^{2t_0}-1)+(\mathrm{e}^{2(T-T_{\hat{k}-1})}-1)\leq \hat{k}(\mathrm{e}^{2t_0}-1)\leq \left(\frac{T}{t_0}+1\right)(\mathrm{e}^{2t_0}-1)\leq 2(\mathrm{e}^{2T}-1).
\end{equation}
We now introduce the weighted norms
\begin{equation}\label{diff-rho_n}
 \|\rho_{n+1}-\rho_n\|_k := \sup_{S^*_T \cap ]T_{k-1},T_k]} | \rho_{n+1} (\o,\theta,\Omega, t) - \rho_n (\omega,\theta, \Omega, t)| \frac{\mathrm{e}^{M_{k}\omega^2}}{{C}_{\O,T}},
 \qquad k=1,...,\kk.
\end{equation}
From \eqref{eq-mk} and \eqref{rhosup} it follows that 
\begin{equation}
 \|\rho_{n+1}-\rho_n\|_k \le 2, \qquad k=1,...,\kk.
\end{equation}
The term $\frac{\mathrm{e}^{M_{k}\omega^2}}{{C}_{\O,T}}$ appearing in the norm introduced in \eqref{diff-rho_n} is useful in order to deal with the term $f_n$ in \eqref{represent}. Indeed, later on we will rely on the following bound	
\begin{equation} \label{diff-phi-rho_n}
	\begin{aligned}
		&|\Phi_{\rho_{n}}(\theta,t)- \Phi_{\rho_{n-1}}(\theta,t)| \leq 
		\\&\leq K \int_\R g(\Omega) \int_\R \int_0^{2\pi} |\sin(\theta-\theta')(\rho_{n} (\omega',\theta', \Omega, t) - \rho_{n-1} (\omega',\theta',  \Omega, t))| d\theta' d\omega' d\Omega \leq 
		\\& \leq K \|\rho_{n}-\rho_{n-1}\|_{k+1}  \int_{\R}\int_0^{2\pi} \mathrm{e}^{-{M_{k+1}}\omega'^2}d\theta' d\omega' \int_\R g(\Omega) {C}_{\O,T}d\O= \frac{2 \pi^{\frac{3}{2}} K G }{\sqrt{M_{k+1}}} \|\rho_{n}-\rho_{n-1}\|_{k+1},
	\end{aligned}
\end{equation} 
for every $\theta \in \R$ and $ t \in ]T_k,T_{k+1}[$, which will be used to estimate $|f_n|$.

In order to prove \eqref{eq-rho_n-to-rho} we claim that the following estimates
\begin{equation}\label{disuguaglianzeq}
 \begin{aligned}
 \|\rho_{n+1}-\rho_n\|_1\leq & \, q \, \|\rho_{n}-\rho_{n-1}\|_1, \\
 \|\rho_{n+1}-\rho_n\|_{k+1} \leq & \, q \, \|\rho_{n}-\rho_{n-1}\|_{k+1} +\overline{C}_{\O,T} \|\rho_{n+1}-\rho_n\|_k,   
 \end{aligned} 
\end{equation}
hold for all $n=2,3,..$ and $k=1,2,...,\kk-1$. Let's denote by $I_1$ and $I_2$ the two integral in the right hand side of \eqref{represent}. By \eqref{eq-Gammatilde-0} and \eqref{estimatesderivate} a direct computation give us
\begin{equation}\begin{aligned}
	|I_1|\leq& \overline{C}_{\O,T}\|\rho_{{n+1}}-\rho_n\|_k \, {C}_{\O,T}\int_{\R} \widetilde \G^\e(\o,t;\x,T_k)\mathrm{e}^{-M_k\xi^2}d\xi \leq\\ \leq & \overline{C}_{\O,T}\|\rho_{{n+1}}-\rho_n\|_k\, {C}_{\O,T}\;\mathrm{exp}\left(-\o^2\frac{M_k}{1+M_k(1+\e)(\mathrm{e}^{2(T_{k+1}-T_k)}-1)}\right)= \\=&\overline{C}_{\O,T}\|\rho_{{n+1}}-\rho_n\|_k \,  {C}_{\O,T}\mathrm{e}^{-M_{k+1}\o^2}.
	\end{aligned}
\end{equation}
In order to estimate $I_2$, we use again \eqref{eq-Gammatilde-0} and \eqref{estimatesderivate}, together with \eqref{diff-phi-rho_n} and the exponential decay of $\p_\o \rho_n$ 
\begin{equation}
	\begin{aligned}
  |I_2| & \leq\frac{2 \pi^{\frac{3}{2}} K  G\overline{C}_{\O,T}}{\sqrt{M_{k+1}}}\|\rho_{n}-\rho_{n-1}\|_{k+1} {C}_{\O,T} \int _{T_k}^t \int_{\R}\frac{\widetilde \G^\e(\o,t;\x,\t)\mathrm{e}^{-\overline{M}\xi^2} }{\sqrt{\tau}}  d\xi d\tau \leq \\ &\leq 2\pi^{\frac{3}{2}} K G\overline{C}_{\O,T}\sqrt{ \frac{t_0}{A}}\|\rho_{n}-\rho_{n-1}\|_{k+1}{C}_{\O,T}\mathrm{e}^{-M_{k+1}\o^2} = q \|\rho_{n}-\rho_{n-1}\|_{k+1}{C}_{\O,T}\mathrm{e}^{-M_{k+1}\o^2},
  	\end{aligned}
\end{equation}
where we have used the inequalities in \eqref{eq-mk}. This concludes the proof of \eqref{disuguaglianzeq}.

We next recall that Lemma 3.4 in \cite{Sp1} states that \eqref{disuguaglianzeq} imply the inequalities \begin{equation}\label{sommaq}
	\|\rho_{n+1}-\rho_{n}\|_k \leq q^{n-2} \sum_{i=0}^{k-1}(\overline{C}_{\O,T}(n-2))^i\|\rho_{3}-\rho_{2}\|_{k-i},
\end{equation}
for all $ n=3,4,...$ and $k=1,...,\kk$. We then find
\begin{equation}
  \|\rho_{n+1}-\rho_{n}\|_{C(\overline{S^*_T})}\leq\sum_{k=1}^{\kk}\|\rho_{n+1}-\rho_{n}\|_k \leq\sum_{k=1}^{\kk}q^{n-2}\sum_{i=0}^{k-1}(\overline{C}_{\O,T}(n-2))^i\|\rho_{3}-\rho_2\|_{k-i}.
\end{equation}
By using the elementary inequality $(\overline{C}_{\O,T}(n-2))^i \le (n-2)^\kk (1+\overline{C}_{\O,T})^{\kk}$ and
\begin{equation}
  \sum_{k=1}^{\kk}\sum_{i=0}^{k-1} \|\rho_{3}-\rho_{2}\|_{k-i} = \sum_{k=1}^{\kk}\sum_{j=1}^k \|\rho_3-\rho_2\|_j \leq \sum_{k=1}^{\kk}\sum_{j=1}^{\kk} \|\rho_3-\rho_2\|_j =\kk \sum_{j=1}^{\kk} \|\rho_3-\rho_2\|_j,	
\end{equation}
we finally have
\begin{equation}
  \|\rho_{n+1}-\rho_{n}\|_{C(\overline{S^*_T})}\leq q^{n-2} (n-2)^\kk (1+\overline{C}_{\O,T})^{\kk}\kk\sum_{j=1}^\kk\|\rho_{3}-\rho_2\|_{j}.
\end{equation}
Then, if we set
\begin{equation}
  V  := (1+ \overline{C}_{\O,T})^{\kk}\kk\sum_{j=1}^\kk\|\rho_{3}-\rho_2\|_{j},
\end{equation}
and we notice that $V$ does not depend on $n$, we eventually obtain
\begin{equation}
    \sum_{n=3}^{+\infty}\|\rho_{n+1}-\rho_{n}\|_{C(\overline{S^*_T})}\leq V \sum_{n=3}^{+\infty}  q^{n-2}(n-2)^\kk = V \sum_{n=1}^{+\infty} q^nn^\kk < \infty.
\end{equation}
This accomplishes the proof of \eqref{eq-rho_n-to-rho}. 

From \eqref{eq-rho_n-to-rho} and \eqref{rhosup} it follows that $\rho$ satisfies the identity \eqref{eq-repr-rho}, it has property {\rm ({\bf E})} in the set $S_T^*$, and 
\begin{equation}
  \rho(\o,\theta,\Omega,0)=\rho_0(\o,\theta,\Omega),\qquad 
  \mathrm{for \;every}  \ (\o,\theta,\Omega) \in \R^3.
\end{equation}
Moreover, $\r$ is a classical solution to $\L_{\O,\r}\r = 0$ in $S_T^*$, thanks to the bounds \eqref{stimay} in Lemma \ref{lemmam}, that hold for $\rho_n$, uniformly with respect to $n$. This completes the proof of Proposition \ref{result}.
\end{proof}

\medskip

We now prove that, under the additional assumption that $g$ has compact support,
there is only one solution satisfying the property {\rm ({\bf E})}.

\begin{proposition}\label{Prop-unique}
Let us suppose that $g$ has support in the set $[-\O_1,\O_1]$, and let $\rho_1$ and $\rho_2$ be two classical solutions to the Cauchy problem \eqref{eq-Cauchy-pbm}. If $\rho_1$ and $\rho_2$ satisfy property {\rm ({\bf E})}, then $\rho_1 =\rho_2$.
\end{proposition}

\begin{proof}
Suppose that $\rho_1$ and $\rho_2$ are two solutions to the Cauchy problem \eqref{eq-Cauchy-pbm}, both verifying property  {\rm ({\bf E})}, with constants $\overline{M},C_{\O,T}$. We define $\overline{\rho}:=\rho_1 - \rho_2$, and we note that the following quantities are finite for every $\t, t \in [0,T[$  
\begin{equation}\begin{aligned}\label{def-phi-tilde}
		\tilde\phi(\t):=\int_\R\max_{\theta \in \R,\O\in [-\O_1,\O_1]}\left|\frac{\overline{\rho}(\o,\theta,\O,\t)}{C_{\O,T}}\right|d\o,\qquad \phi(t):=\sup_{0\le \t\leq t}\tilde \phi(\t).
\end{aligned}\end{equation}
Moreover, $\overline{\rho}$ is a classical solution to
\begin{equation}
	\begin{dcases}
  \p^2_{\omega \omega} \overline{\rho}+ \omega \p_\omega \overline{\rho} - \omega \p_\theta \overline{\rho} +  \overline{\rho} - \p_t \overline{\rho}=- \left((\Phi_{\rho_2}-\Phi_{\rho_1}) \p_\omega \rho_2 - \Phi_{\rho_1}\p_\omega (\rho_2-\rho_1)\right), \\
		\overline \rho(\o,\theta,\Omega,0)=0.
	\end{dcases}
\end{equation} 
Arguing as in the proof of Theorem \ref{Th-main}, we see that $\overline{\rho}$ can be represented as
\begin{equation}\label{rep-diff}
	\overline{\rho}(\o,\theta,\Omega,t) = \int_0^t\int_{\R^2}\overline{\G}(\o,\theta,t;\xi,\eta,\t)\left((\Phi_{\rho_2}-\Phi_{\rho_1}) \p_\xi \rho_2 - \Phi_{\rho_1}\p_\xi (\rho_2-\rho_1)\right)(\xi,\y,\O,\t)d\xi d\y d\t,
\end{equation}
where $\overline{\G}$ is the fundamental solution of \begin{equation}
	\overline{\L} =  \p^2_{\omega \omega}+ \omega \p_\omega  - \omega \p_\theta  +  1 - \p_t.
\end{equation}
Our aim is to obtain an estimate for $\phi(t)$, using the representation formula \eqref{rep-diff}, which will implies $\phi(t_0)=0$ for a sufficiently small $t_0$. Moreover, $t_0$ only depends on the operator $\overline{\L}$ and on the quantities appearing in property  {\rm ({\bf E})}, then the argument can be repeated in the interval $[t_0, 2 t_0]$, then in $[2t_0, 3 t_0]$, and so on, hence we conclude that $\overline{\rho}(\o,\theta,\Omega,t) = 0$ for every $(\o,\theta,\Omega,t) \in \R^3 \times [0,T[$.

For this purpose, we first recall the constant $G$ introduced in \eqref{G}, and we note that the following estimate
\begin{equation} \label{}
	\begin{aligned}
		|\Phi_{\rho_{2}}(\y,\t)-& \Phi_{\rho_{1}}(\y,\t)| \leq 
		\\& K \int_\R g(\Omega) \int_\R \int_0^{2\pi} |\sin(\y-\theta')(\rho_{2} (\omega',\theta', \Omega, \t) -
		 \rho_{1} (\omega',\theta',  \Omega, \t))| d\theta' d\omega' d\Omega \leq \\ &\qquad \qquad \qquad \qquad \qquad \qquad \qquad \qquad\qquad \qquad\leq 2\pi K G\phi(t),\qquad \y \in \R ,\t \in [0,t],
	\end{aligned}
\end{equation} 
plainly follows from the definition of $\Phi_{\rho_{1}}$ and $\Phi_{\rho_{2}}$.
Then, from this bound, the exponential decay property of the derivative $\p_\o \rho_2$ and the bound \eqref{eq-bound-Phi} applied to $\Phi_{\rho_1}$ we find that
\begin{equation}\begin{aligned}
	\int_\R|\overline{\rho}(\o,\theta,&\Omega,t)|d\o\leq\\& 2\pi K G \phi(t) \int_0^t\int_{\R^3}\overline{\G}(\o,\theta,t;\xi,\eta,\t) \frac{C_{\O,T}}{\sqrt{\t}}\mathrm{e}^{-{\overline{M}\xi\,^2}}d\y d\o  d\xi  d\t+\\&(|\Omega| + K)\int_0^t\int_{\R^3}|\p_\xi \overline{\G}(\o,\theta,t;\xi,\eta,\t)| (\rho_2-\rho_1)(\xi,\y,\O,\t)d\w d\y d\xi  d\t=I_1+I_2.
\end{aligned}\end{equation}
By a direct computation we find that
\begin{equation}
	I_1 \le C_{\O,T}\;\frac{\pi^{\frac{3}{2}} K G }{\sqrt{\overline{M}}}\sqrt{t} \phi(t),
\end{equation}
Moreover, by \eqref{estimatesderivate}, together with definition \eqref{def-phi-tilde} and bounds \eqref{expr-c}, \eqref{explicit-c}, we obtain
\begin{equation}
	I_2 \le C_{\O,T}\overline{C}\frac{\sqrt{t}}{2}\phi(t),
\end{equation}
for some positive constant $\overline{C}$ that depends only on $\O_1$. Thus, since the previous estimates gives us 
\begin{equation}\label{upper-phi}
	\phi(t)\leq \phi(t)\left(\frac{\pi^{\frac{3}{2}} K G }{\sqrt{\overline{M}}}\sqrt{t} + \overline{C}\frac{\sqrt{t}}{2}\right),
\end{equation}
we can state that there exists $t_0$ such that, if $T\leq t_0$, we have $\phi(t)\leq c\,\phi(t)$ for some $c<1$, and therefore $\phi(t)=0$.

Finally, by iterating this arguments in time intervals $[kt_0,(k+1)t_0]$ the thesis follows.
\end{proof}

\medskip\medskip
\begin{proof}{\sc of Theorem \ref{Th-main}}
The proof of the Theorem follows from the results we have established in Lemma \ref{lemmam}, Proposition \ref{result} and Proposition \ref{Prop-unique}. The continuity dependence of the solution with respect to $\O$ is proved in \cite{DPa}.
\end{proof}

	\section{Finite difference method}

The purpose of this Section is to develop a numerical scheme in the spirit of the analytical framework introduced in Section 2. We propose therefore a finite difference scheme where we discretize the Lie derivative instead of the classical derivatives with respect to $\theta$ and $t$. In particular, we are following the pattern described in \cite{DFP} and \cite{PM}, where an accurate analysis concerning the discretization of degenerate Kolmogorov operators has been made.  The discretization we would like to carry out would make us approximate $Y\rho =  \omega \p_\omega\rho   - \omega \p_\theta\rho - \p_t\rho $ with the following quotient
\begin{equation}
 \frac{\rho(\o \mathrm{e}^\d,\omega - \omega \mathrm{e}^\d-\theta,t-\d)-\rho(\o,\theta,\O,t)}{\d}. 
\end{equation}
The main difficulty when setting an appropriate grid is represented by the factor $\mathrm{e}^\d$, which arises from the term $\o\p_\o$ in the Lie derivative. This problem can be overcome by observing that, thanks to the continuity w.r.t. the derivative $\p_\o$, the solution is also continuous with respect to the Lie derivative $Y_r = -\omega\p_\theta-\p_t$, whose discretization is easier.

In our analysis we assume that the natural frequency distribution $g$ is compactly supported in the set $[-\O_1,\O_1]$, so that the solution we approximate is unique, according to Proposition \ref{Prop-unique}.

We list here the various approximations that we use for the numerical scheme:
\begin{equation}
  \begin{aligned}
  &\p^{2}_{\o\o} \rho(\o,\theta,\O,t) = \frac{\rho(\o+\Delta\o,\theta, \O,t) -2\rho(\o,\theta,\O,t)+\rho(\o-\Delta\o,\theta, \O,t)}{(\Delta\o)^2} + o(\D\o)^2,\\
  & \p_{\o} \rho (\o,\theta,\O,t)= \frac{\rho(\o+\Delta\o,\theta, \O,t) -\rho(\o-\Delta\o,\theta, \O,t)}{2(\Delta\o)}+ o(\D\o)^2,\\
   &Y_r \rho(\o,\theta,\O,t) =  \frac{\rho(\o,\theta-\o\D t,\O,t-\D t) - \rho(\o,\theta,\O,t)}{\D t} + o({\D t}).
  \end{aligned}
\end{equation}
Finally, we approximate the nonlinear term by
\begin{equation}
  -\O- K\int_{-\O_1}^{\O_1}\int_0^{2\pi}\int_{-G_\o}^{G_\o}\rho(\o',\theta',\O,t-\D t)\sin(\theta-\theta')g(\O') d\o'd\theta'd\O' +  o({\D t}).
\end{equation}
This approximation introduces the constant $G_\o$ to overcome the problem of an unlimited interval in $\o$. We observe that, thanks to the exponential decay, we can make this last approximation truly negligible. Finally, we approximate this integral by the method of rectangles, and we  denote as $\tilde\Phi_\rho$ this last approximation. As in the classical case, the various remainders depend on the $L^\infty$ norm of the solution and of some of its derivatives.

\medskip In the following, in order to provide a comprehensive analysis of the model, we will include the cases where the inertial term $m$ and the noise term $D$ are not necessarily equal to one, in order to investigate the changes in the solution as these parameters vary.

We define the approximate operator $\L_a$ as 
\begin{equation}\begin{aligned}
  \L_a \rho(\o,\theta,&\O,t) := \\
  = &\frac{D}{m^2}\frac{\rho(\o+\Delta\o,\theta-\o\D t, \O,t) -2\rho(\o,\theta-\o\D t,\O,t)+\rho(\o-\Delta\o,\theta-\o\D t, \O,t)}{(\Delta\o)^2}  \\
  +&\frac{1}{m} \left(\o +\tilde\Phi_\rho(\theta-\o\D t,\O,t)\right) \frac{\rho(\o+\Delta\o,\theta-\o\D t,, \O,t) -\rho(\o-\Delta\o, \theta-\o\D t,  \O,t)}{2(\Delta\o)} \\
   +& \frac{\rho(\o,\theta-\o\D t,\O,t) - \rho(\o,\theta,\O,t+\D t)}{\D t} + \frac{1}{m}\rho(\o,\theta-\o\D t,\O,t).
\end{aligned}\end{equation}
The previous approximations make $\L_a$ consistent with $\L_\rho$; indeed when $G_\o=+\infty$  we have:
\begin{equation}
  \L_\rho\rho(\o,\theta,\O,t)-\L_a\rho(\o,\theta,\O,t)=o((\D\o)^2+{\D t}).
\end{equation}
In order to define the corresponding grid, let us examine the domain in which we  solve the problem: as anticipated above, we suppose that $\o$ belongs to the bounded interval $[-G_\o,G_\o]$; moreover, by the periodicity in $\theta$ we can suppose that this variable belongs to $[0,2\pi]$ and then we work with periodic conditions. Thus, the natural grid where we discretize our problem is
\begin{equation}\begin{aligned}
  G_r = &\left\{i{\D\o},j{\D\o}{\D t},  k \D \O, n {\D t},\right.\\&\left. -\frac{G_\o}{\D \o}\leq i \leq\frac{G_\o}{\D \o} ,\;\;0 \leq j \leq \frac{2 \pi }{\D \o \D t},\; 0\leq n \leq \frac{T}{\D t},\; -\frac{\O_1}{\D \O} \leq k\leq \frac{\O_1}{\D \O} \right\},\\
  G_0 = & G_r \cap \{n=0\},\;\;\;\;\;G_b = G_r \cap \left\{|i|=\frac{G_\o}{\D \o}\right\},\;\;\;\;\;G_{in} =  G_r \setminus (G_0 \cup G_b).
\end{aligned}\end{equation}
We conclude with one last remark. As we have shown in Lemma \ref{lemmam}, if the initial datum $\rho_0$ is normalized for every $\O$, then the same condition holds true for the solution $\rho$; in order to achieve this, we introduce the following normalization condition
\begin{equation}\label{normdisc}
(\D \o)^2 (\D t) \sum_{\underline{i},\underline{j}}\rho^n_{\underline{i},\underline{j},k}=1,\qquad\forall\;\; 0\leq n \leq \frac{T}{\D t} ,\;-\frac{\O_1}{\D \O} \leq k\leq \frac{\O_1}{\D \O},
\end{equation} 
The approximate Cauchy problem is
\begin{equation}\begin{cases}
  \L_a \rho = 0,\;\;\;\;\;&\mathrm{in}\; G_{in},\\
  \rho = \rho_0,\;\;\;\;\;& \mathrm{in}\; G_0,\\
  \rho = 0,\;\;\;\;\;& \mathrm{in}\; G_b,\\
\end{cases}\end{equation}
with $\rho$ that verifies \eqref{normdisc} and such that $\rho(\o,\theta,\O,t)= \rho(\o,\theta+2\pi,\O,t)$. While the first two equations are the classical approximation of the Cauchy problem \eqref{eq-Cauchy-pbm}, the last is intended to approximate the exponential decay condition.

As in the heat equation case, we need to introduce some stability conditions.
\begin{theorem}
  Suppose that the grid $G_r$ is such that
  \begin{equation}\label{condition}
  \D\o\leq \frac{\sqrt{2D}}{m},\qquad\D t \leq \frac{m^2(\D \o)^2}{2D - m (\D \o)^2},\qquad G_\o \leq \frac{2D}{m\D \o}-\O_1-K,
  \end{equation}
 then every solution to 
\begin{equation}\begin{cases}
  \L_a \rho \leq 0,\;\;\;\;\;&\mathrm{in}\; G_{in},\\
  \rho \geq 0,\;\;\;\;\;& \mathrm{in}\; G_0,\\
  \rho = 0,\;\;\;\;\;& \mathrm{in}\; G_b,
\end{cases}\end{equation}
that verifies \eqref{normdisc} and such that $\rho(\o,\theta,\O,t)= \rho(\o,\theta+2\pi,\O,t)$, is nonnegative in $G_r$.
\end{theorem}
\begin{proof}
The proof runs by induction. Let us set
\begin{equation}
  \rho^{n}_{i,j,k} = \rho(i \D \o,j\D \o \D t, k \D \O, n \D t).
\end{equation}
As $\rho\geq 0$ in $G_0$, the condition holds for $n=0$. Let now suppose true the claim for $n$.

We notice that $\L_a\rho \geq 0$ can be read as
\begin{equation}\begin{aligned}
   \rho^{n+1}_{i,j,k}&\geq \left(1-\frac{2D\D t}{m^2(\D\o)^2} +\frac{\D t}{m}\right) \rho^{n}_{i,j-i,k}\\&+ \left(\frac{D\D t}{m^2(\D\o)^2}-\frac{\D t}{2\D\o m}(i\Delta\o+\tilde	\Phi_\rho ((j-i)\D\o,k\D\O,n\D t))\right)\rho^{n}_{i-1,j-i,k} \\
   &+\left(\frac{D\D t}{m^2(\D \o)^2} + \frac{\D t}{2\D\o m}(i\D\o+\tilde\Phi_\rho ((j-i)\D\o,k\D\O,n\D t)\right)\rho^{n}_{i+1,j-i,k}, 
\end{aligned}\end{equation}
where 
\begin{equation}
  \tilde\Phi_\rho((j-i)\D\o,k\D\O,n\D t)=-k\D\O-K\D \O (\D \o)^2 (\D t) \sum_{\underline{i},\underline{j},\underline{k}}\rho^n_{\underline{i},\underline{j},\underline{k}}g(\underline{k})\sin(j-i+\underline{i}).
\end{equation}
We employ the normalization condition \eqref{normdisc} and the normalization of $g$ to note that $|\tilde\Phi_\rho|\leq K+\O_1$, then; if \eqref{condition} holds, it follows  $\rho^{n+1}_{i,j,k}\geq 0$.
\end{proof}
\\\\We illustrate now some numerical results. The following quantities have been analysed
\begin{equation}
  \begin{aligned}
  &|r(t)| = \left|\int_0^{2 \pi}\int_\R\int_\R \mathrm{e}^{\mathrm{i}\theta}\rho(\o,\theta,\Omega,t)g(\Omega)d\Omega d\o d\theta\right|,\\
  &|s(t)| = \left|\int_0^{2 \pi}\int_\R\int_\R \mathrm{e}^{\mathrm{i}\o}\rho(\o,\theta,\Omega,t)g(\Omega)d\Omega d\o d\theta\right|.
  \end{aligned}
\end{equation}
The first one represents the phase coherence, and gives us information about the phase synchrony, while the latter is the analogous for frequencies.

We have considered identical oscillators with natural frequency $\O=0$, and with the following initial conditions
\begin{equation}
  \rho_0(\o,\theta)= \frac{\mathrm{e}^{-{\o^2}}(\sin(\theta)+1)}{N(\o+1)},\;\;\;\;\;N = \int_\R\int_{0}^{2\pi}\frac{\mathrm{e}^{-\o^2}(\sin(\theta)+1)}{(\o+1)}d\o d\theta.
\end{equation}
As we could have expected, we see in Figure \ref{1} that an increase in the coupling strength $K$ leads to a greater phase synchrony, while in Figure \ref{2} we see that the frequency synchronization seems to be asymptotically independent of $K$. 

\begin{figure}[h]
 \begin{minipage}{0.48\linewidth}
  \centering
  \includegraphics[width=8cm, height=8cm]{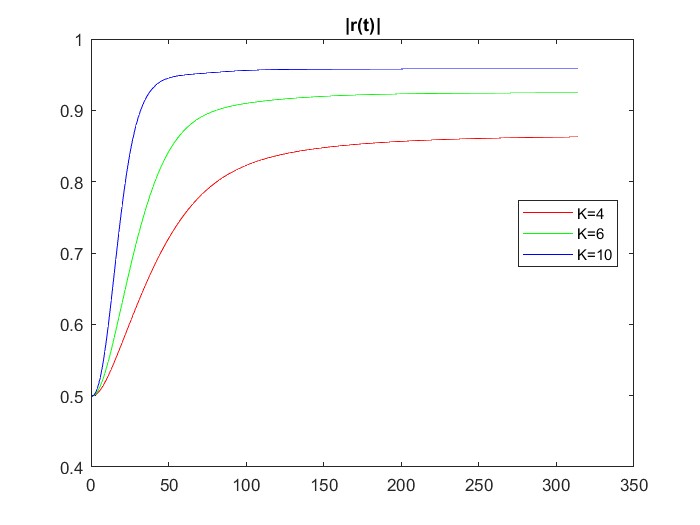}
  \caption{Time evolution of the phase coherence, $m=1,D=1,T=10,\Delta t = 0.0317$.}
  \label{1}
 \end{minipage}\hspace{0.2cm}
 \begin{minipage}{0.48\linewidth}
  \centering
  \includegraphics[width=8cm, height=8cm]{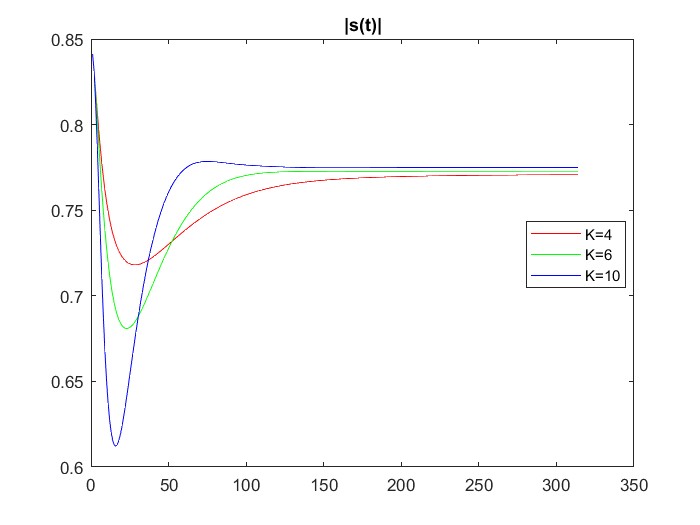}
  \caption{Time evolution of the frequency coherence, $m=1,D=1,T=10,\Delta t = 0.0317$.}
  \label{2}
 \end{minipage}
\end{figure} 

Conversely, changes in the inertial term $m$ seem to have the opposite nature: indeed, in Figure \ref{3} we see that the asymptotic behaviour of phase synchronization does not change; on the other side, in Figure \ref{4} we notice that the behaviour of frequency synchronization significantly depends on this parameter, as one can observe by comparing this case with the one in Figure \ref{2}. In particular, we see that the asymptotic behaviour of the frequency coherence increases as $m$ increases.

\begin{figure}[!ht]
 \begin{minipage}{0.48\linewidth}
  \centering
  \includegraphics[width=8cm, height=8cm]{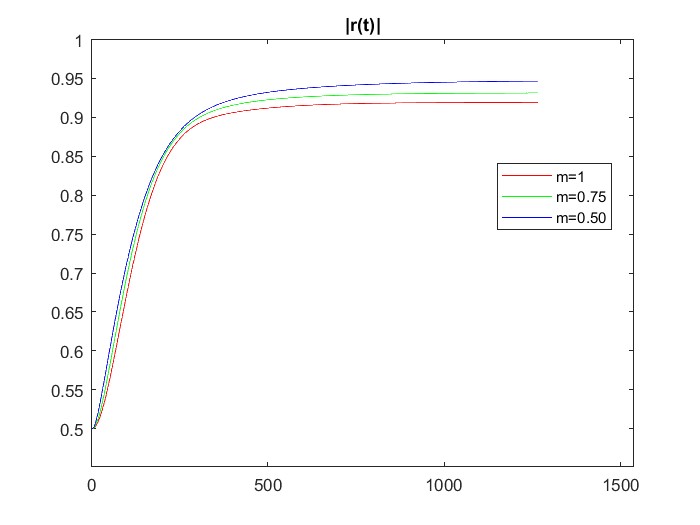}
  \caption{Time evolution of the phase coherence, $D=1,K=6,T=10, \Delta t = 0.0079$.}
  \label{3}
 \end{minipage}\hspace{0.2cm}
 \begin{minipage}{0.48\linewidth}
  \centering
  \includegraphics[width=8cm, height=8cm]{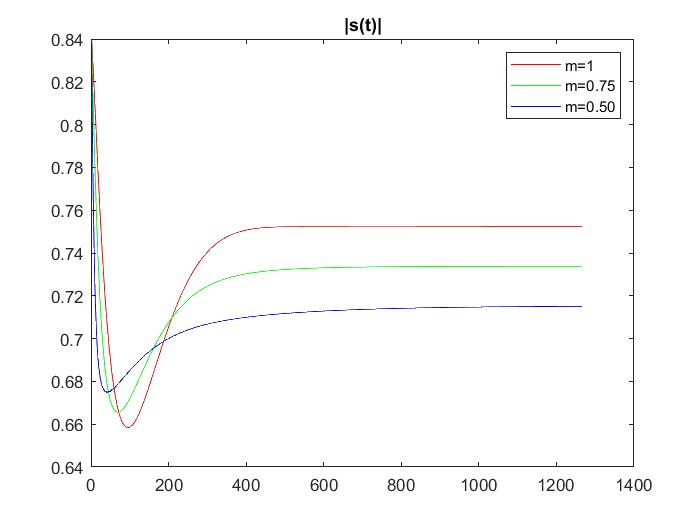}
  \caption{Time evolution of the frequency coherence, $D=1,K=6,T=10, \Delta t = 0.0079$.}
  \label{4}
 \end{minipage}
\end{figure}
Regarding changes in the noise term $D$, we see  what is expected in the classical Kuramoto model. We remind indeed (see for instance \cite{S}) that in the original model ($m=D=0$) whenever we have identical oscillators the phase synchrony is total (\emph{i.e.} $|r|=1$). Moreover, these oscillators synchronize at the same frequency (therefore $|s|=1$). Our numerical results are in good agreement with the  theoretical ones: indeed, in  Figure \ref{5} and Figure \ref{6} we see that the phase and frequency coherence increase for decreasing $D$.
\begin{figure}[!ht]
 \begin{minipage}{0.48\linewidth}
  \centering
  \includegraphics[width=8cm, height=8cm]{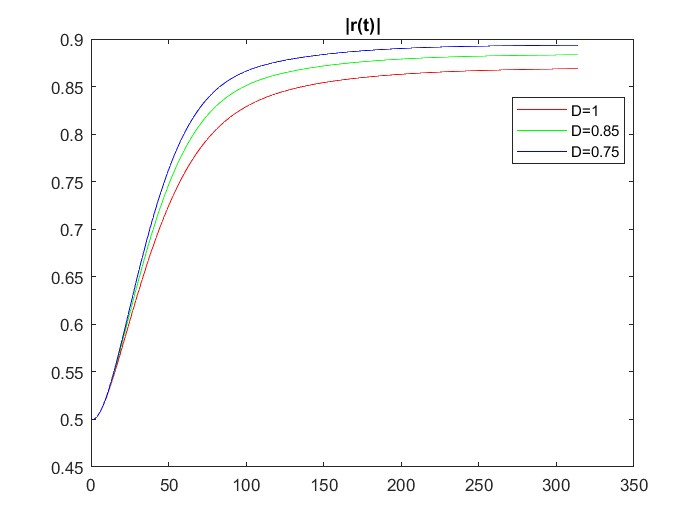}
  \caption{Time evolution of the phase coherence, $m=1,K=6,T=10,\Delta t = 0.0317$.}
  \label{5}
 \end{minipage}\hspace{0.2cm}
 \begin{minipage}{0.48\linewidth}
  \centering
  \includegraphics[width=8cm, height=8cm]{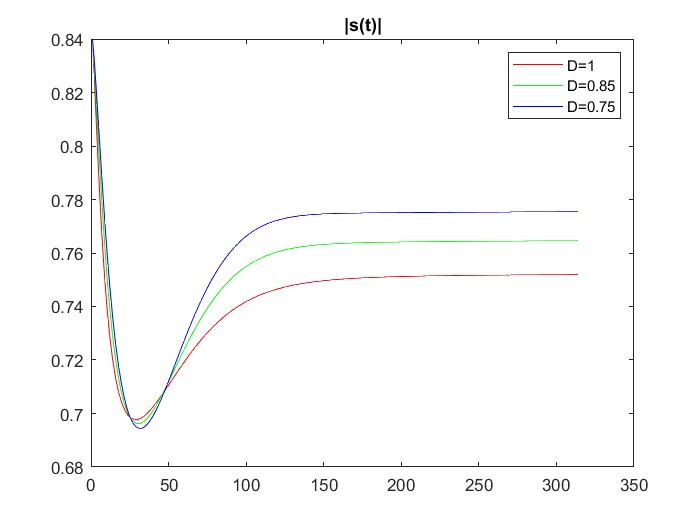}
  \caption{Time evolution of the frequency coherence, $m=1,K=6,T=10,\Delta t = 0.0317$.}
  \label{6}
 \end{minipage}
\end{figure}

\begin{figure}[!t]
 \begin{minipage}{0.48\linewidth}
  \centering
  \includegraphics[width=8cm, height=8cm]{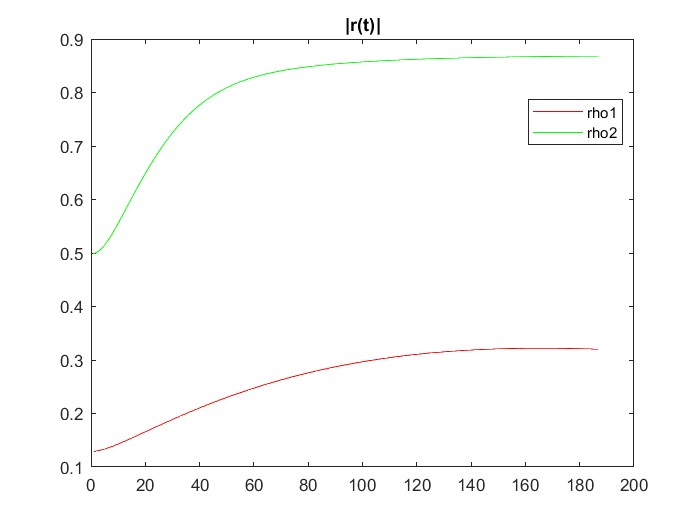}
  \caption{Time evolution of the phase coherence, $m=1,D=1,K=4,T=10,\\\Delta t=0.053$.}
  \label{rho1}
 \end{minipage}\hspace{0.2cm}
 \begin{minipage}{0.48\linewidth}
  \centering
  \includegraphics[width=8cm, height=8cm]{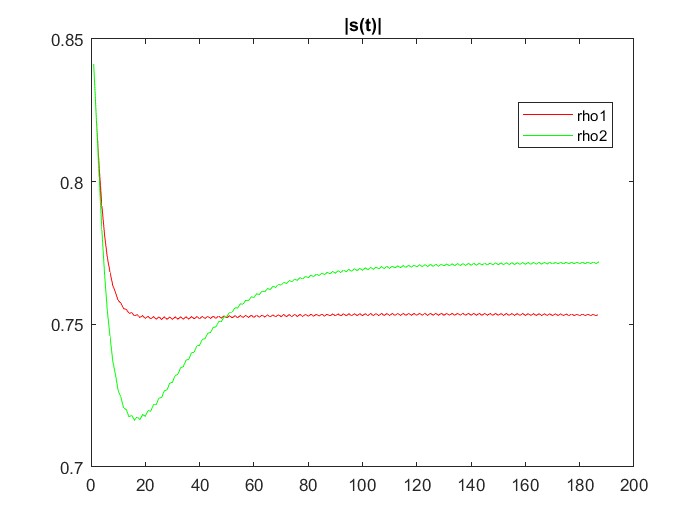}
  \caption{Time evolution of the frequency coherence, $m=1,D=1,K=4,T=10,\\\Delta t=0.053$.}
  \label{rho2}
 \end{minipage}
\end{figure} 
Finally, in Figure \ref{rho1} and Figure \ref{rho2} we compare two different initial conditions
\begin{equation}
  \begin{aligned}
    &\rho_{0,1}(\o,\theta)= \frac{\mathrm{e}^{-\o^2}(\sin(\frac{\theta}{2})+1)}{N_1(\o+1)},\;\;\;\;\;N_1 = \int_\R\int_{0}^{2\pi}\frac{\mathrm{e}^{-\o^2}(\sin(\frac{\theta}{2})+1)}{(\o+1)}d\o d\theta,\\
    &  \rho_{0,2}(\o,\theta)= \frac{\mathrm{e}^{-\o^2}(\sin(\theta)+1)}{N_2(\o+1)},\;\;\;\;\;N_2 = \int_\R\int_{0}^{2\pi}\frac{\mathrm{e}^{-\o^2}(\sin(\theta)+1)}{(\o+1)}d\o d\theta,
  \end{aligned}
\end{equation}
and we observe two quite different behaviours, suggesting great susceptibility to initial conditions.

We conclude with one last remark. These numerical results are in agreement with the ones in \cite{AS}, where the finite size system \eqref{spig} was analysed. The point we are making is that we have used a different method here: indeed, in \cite{AS} the authors used a Monte Carlo method (thus a stochastic method) in the finite size case (with $N=30000$ oscillators), while we have employed a finite difference method to analyse \eqref{ultraparabolica}, that is the limit of $\eqref{spig}$ as $N \rightarrow \infty$.

\end{document}